\documentclass[10pt]{article}
\usepackage{amssymb,amsthm,amsmath}
\usepackage{color,graphicx}
\usepackage{authblk}
\usepackage[dvipsnames]{xcolor}
\usepackage{tikz,cancel}
\usepackage{mathrsfs}
\usetikzlibrary{patterns}
\usepackage{mathtools}
\usepackage{caption}
\usepackage{lineno}
\usepackage{hyperref}
 \hypersetup{linkcolor=blue, colorlinks=true,citecolor = red}
\usepackage[capitalise]{cleveref}
\usepackage{upref}
\usepackage{changes}
\usepackage{mathrsfs}

\newtheorem{corollary}{Corollary}[section]
\newtheorem{definition}[corollary]{Definition}

\newtheorem{lemma}[corollary]{Lemma}
\newtheorem{proposition}[corollary]{Proposition}
\newtheorem{remark}[corollary]{Remark}
\newtheorem{theorem}[corollary]{Theorem}
\numberwithin{equation}{section}

\renewcommand{\leq}{\leqslant}

\renewcommand{\geq}{\geqslant}

\makeatletter
\def\section{\@startsection {section}{1}{\z@}{-3.5ex plus -1ex minus
    -.2ex}{2.3ex plus .2ex}{\large\bf}}

\makeatother

\begin{document}

\title{\bf LQR control for a system\\ describing the interaction between a floating solid and the surrounding fluid  }
\author[1]{Marius Tucsnak}
\affil[1]{Institut de Math\'{e}matiques de Bordeaux UMR 5251, Universit\'{e} de Bordeaux/Bordeaux INP/CNRS, 351 Cours de la Lib\'{e}ration, 33 405 TALENCE, France} 
\author[2]{Zhuo Xu}
\affil[2]{Institut de Math\'{e}matiques de Bordeaux UMR 5251, Universit\'{e} de Bordeaux/Bordeaux INP/CNRS, 351 Cours de la Lib\'{e}ration, 33 405 TALENCE, France}
\date{Dedicated to professor for H\'el\`ene Frankowska for her 70th anniversary}

\maketitle

\begin{abstract}
This paper studies an infinite time horizon LQR optimal control problem for a system describing, within a linear approximation, the vertical oscillations of a floating solid, coupled to the motion of the free boundary fluid on which it floats. The fluid flow is described by a viscous version of the linearized Saint-Venant equations (shallow water regime). The major difficulty we are facing is that
the domain occupied by the fluid is unbounded so that the system is not exponentially stable. This fact firstly raises challenges in proving the wellposedness, requiring the combined use of analytic semigroup theory and of an interpolation technique. The main contribution of the paper is that we show that, in spite of the lack of exponential stabilizability, we can define a wellposed LQR problem for which a Riccati based approach to design feedback controls can be implemented.
\vskip 0.2in
\par


\end{abstract}

\tableofcontents

\section{Introduction}\label{sec_very_first}
\pagenumbering{arabic}\setcounter{page}{1}
\vspace{0.5cm}

In this paper we consider a coupled PDEs-ODEs system describing the motion of a floating solid in a free surface viscous fluid which fills the remaining part of the space. The considered system
can be seen as a simplified model for a point absorber, one on the most popular models of wave energy converters, see Bocchi \cite{p4}, Falnes and Johannes \cite{p9}, Haak, Maity, Takahashi and Tucsnak \cite{p13}. More generally, problems  describing the interaction of floating or immersed bodies with the fluid surrounding them have been studied, in particular, in John \cite{p14}, \cite{p15}, Lannes \cite{p17}, Maity, Mart\'{i}n, Takahashi and Tucsnak \cite{p19}, Maity and Tucsnak \cite{p20}.

In the present work, the fluid flow is supposed to be viscous and to obey the shallow water regime. The corresponding system (without inputs) has been considered in \cite{p19} in the case when the fluid-solid system is confined in a bounded container. The novelty we bring in this work is, besides introducing an input function, which we consider the situation in which the fluid is not bounded (of obvious interest when we aim at describing a point absorber situated in the ocean) and that we discuss an optimal control problem, with an input acting on the solid. One of the main difficulties which have to be tackled when solving this problem is that the semigroup corresponding to the system with no inputs is not exponentially stable. This type of control is of interest for wave energy converters in order to synchronize the oscillations of the rigid floating object and of the incoming waves, see Korde and Ringwood \cite{korde2016hydrodynamic}.

More precisely, we are concerned with an infinite horizon LQR optimal control problem for the system obtained by linearizing the equations which describe the interaction of a floating solid with the fluid on which it floats. This fluid is assumed to be infinite in the horizontal direction. This fact contrasts with the situation described in the previous paper \cite{p19}, where the fluid was supposed to be in a bounded container. The extension of this model to an unbounded fluid was first considered in Vergara-Hermosilla, Matignon and Tucsnak \cite{p11}, where input-output stability questions have been investigated.

To be more specific, we suppose that the floating solid has vertical walls and it undergoes only vertical motions, see Figure \ref{The body floating in an unbounded fluid}. In this case the projection of the solid on the fluid bottom (supposed to be horizontal) does not depend on time. We denote by $\mathcal{I}=[-a,a]$, with $a>0$, the projection on the fluid bottom of the solid domain and we set $\mathcal{E}=\mathbb{R} \verb|\| [-a,a]$. We designate by $\mu>0$ the viscosity coefficient of the fluid. The solid is supposed to have unit mass and it is constrained to move only in the vertical direction. We denote by $q(t,x)$ the flux at instant $t$ through the section $x$, by $h(t,x)$  the height of the free surface of the fluid at the point of abscissa $x$ at instant $t$, whereas $H(t)$ stands for  the distance from the bottom of rigid body to the sea bottom at instant $t$. Moreover, $p(t,x)$, designs the pressure under the rigid body (which can be seen as a Lagrange multiplier, see \cite[Section 2]{p19}. The input signal $u$ designs the variation with respect to time of the vertical force exerted by the controller on the rigid body. With the above notation, the considered system can be written, see \cite{p19}, \cite{p11}, in the form:
\begin{align}
	&\frac{\partial h}{\partial t} +\frac{\partial q}{\partial x}=0 &(& t>0, \;x\in \mathbb{R}),\label{A1}\\
	&\label{conversation_of_monmentun}\frac{\partial q}{\partial t}+\frac{\partial h}{\partial x}-\mu \frac{\partial^2 q}{\partial x^2}=0 &(& t>0, \;x\in \mathcal{E}),\\
	&\label{AHh}H(t)=h(t,x) &(& t>0,\; x\in \mathcal{I}),\\
	&\label{p=0}p(t,x)=0 &(&t>0,\;x\in\mathcal{E}),\\
	&\label{jump_condition_1} h\left(t,-a^-\right)-\mu \frac{\partial q}{\partial x}(t,-a^-)=p\left(t,-a^+\right)+H(t)- \mu \frac{\partial q}{\partial x}\left(t,-a^+\right) &(&t >0),
\end{align}
\begin{align}
	&\label{jump_condition_2}h\left(t,a^+\right)-\mu \frac{\partial q}{\partial x}\left(t,a^+\right)=p\left(t,a^-\right)+H(t)- \mu \frac{\partial q}{\partial x}(t,a^-) &(&t >0),\\
	&\label{newton_second_law}\ddot{H}(t)=\int_{-a}^{a} p(t,x) {\rm d}x+ u(t) &(&t>0),\\
	&\label{theconnectionofHp}\frac{\partial q}{\partial t}+\frac{\partial p}{\partial x}=0 &(& t>0,\; x\in \mathcal{I}),\\
	&q(t,-a^-)=q\left(t,-a^+\right)=q(t,-a),\quad  q\left(t,a^+\right)=q(t,a^-)=q(t,a) &(&t>0),\\
	&\label{A2}\lim_{\vert x\vert \rightarrow \infty}q(t,x)=0  &(&t>0),\\
& H(0)=H_0,\qquad \dot{H}(0)=G_0,
\\ &h(0,x)=h_0(x) &(&x\in \mathcal{E}),\\
&q(0,x)=q_0(x)  &(&x\in \mathbb{R}),
\label{A_Last}
\end{align}
where
\begin{equation}
	f\left(a^+\right)=\lim_{ x \rightarrow a \atop
		x>a}f(x),\qquad f(a^-)=\lim_{ x \rightarrow a \atop
		x<a}f(x).\notag
\end{equation}
Our first main result concerns the well-posedness of above system.

\begin{theorem}\label{th_exist_1}
Let $T>0$.	Then for every $\left[\begin{gathered}
	H_0\
    G_0\
	h_0\
	q_0
\end{gathered}\right]^\top \in \mathbb{R}^2\times H^1(\mathcal{E})\times H^1(\mathbb{R}) $ satisfying the compatibility condition	\begin{equation}\label{compatibility_condition}
	q_0(x)=-G_0x+\frac{q_0(a)+q_0(-a)}{2}\qquad\quad (x\in \mathcal{I}),
\end{equation}
 and every $u\in L^2(0,\infty)$, there exists a unique solution $\left[\begin{gathered}
 	H\
 	h\
 	q\
 	p
 \end{gathered}\right]^\top $ of \eqref{A1}-\eqref{A_Last} satisfying
	\begin{align}
	\label{the_regularity}	&H\in H^2(0,T),\quad h\in H^1\left((0,T);\; H^1(\mathcal{E})\right), \quad p\in L^2\left((0,T);\;\mathcal{P}_2(\mathcal{I})\right) ,\\
	\label{the_regularity2}	& q\in H^1\left((0,T);\; L^2(\mathbb{R})\right)\cap C\left([0,T];\;H^1(\mathbb{R})\right),\quad q\in H^1\left((0,T);\; \mathcal{P}_1(\mathcal{I})\right),\\
	\label{the_regularity3}	&q\in L^2\left((0,T);\;H^2(\mathcal{E})\right).
	\end{align}
	where $\mathcal{P}_1(\mathcal{I})$ and $\mathcal{P}_2(\mathcal{I})$  stand for  the spaces formed by first and second degree polynomials on $\mathcal{I}$, respectively. Moreover, this solution depends continuously on the control function and on the initial data.
\end{theorem}

\begin{figure}
		\includegraphics[width=15cm, height=8cm]{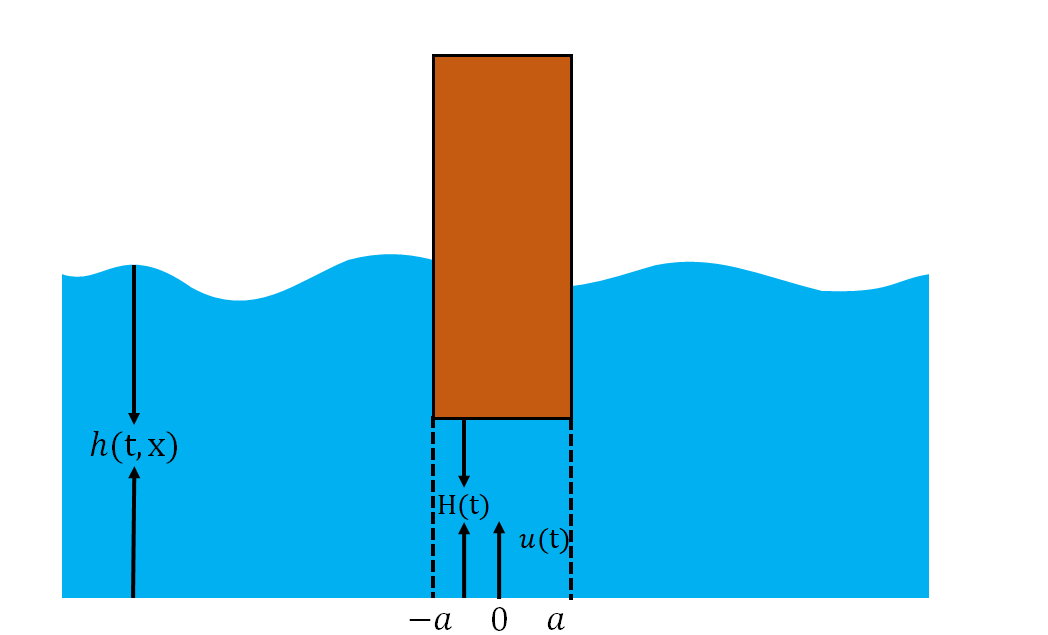}
		\caption{The body floating in an unbounded fluid}\label{The body floating in an unbounded fluid}
\end{figure}

We next consider the optimal control problem associated to the system \eqref{A1}-\eqref{A_Last} which consists in minimizing the cost functional $J: L^2(0,\infty)\to \mathbb{R}\cup \{+\infty\}$ defined by
\begin{gather}\label{defineJu}
	J(u) = \int_{0}^{\infty} \left(\left\vert u(t)\right\vert^2 + \left\vert\dot{H}(t)\right\vert^2\right)\, {\rm d}t \qquad\qquad\left(u\in L^2(0,\infty)\right).
\end{gather}

At this stage, our second main result can be stated as follows:

\begin{theorem}\label{thA1}
	For every initial data satisfying the assumptions in Theorem \ref{th_exist_1}, there exists a unique $\bar u\in L^2(0,\infty)$ such that $J(\bar u)<J(u)$ for every $u\in L^2(0,\infty)$, $u\neq \bar u$. Moreover, $\bar u$ can be written in feedback form.
\end{theorem}

A more detailed version of the above result, including details on the feedback design, will be given, after introducing the necessary notation, in Theorem \ref{thA3}.

As for wellposedness, the main difficulty in the proof of Theorem \ref{thA1} is that system \eqref{A1}-\eqref{A_Last} is not exponentially stable and even not exponentially stabilizable. This fact does not allow the application of the standard theory of infinite horizon LQR problems. Nevertheless, we show below that our system is strongly stable and output stabilizable, which entitles us to use the theory developed by Curtain and Oostveen \cite{curtain1999necessary} for this type of systems.

The remaining part of this work is organized as follows. In Section \ref{backgroundofLqQ} we provide some necessary background on LQR problems for infinite dimensional systems. In Sections \ref{section2} and \ref{sec_well_new} we give the proof of Theorem \ref{th_exist_1}. Finally, Section  \ref{section4} is devoted to the proof of Theorem \ref{thA1}.

\section{Some background on LQR problems}\label{backgroundofLqQ}

As shown below in Section \ref{section4}, equations \eqref{A1}-\eqref{A_Last} do not define an exponentially stabilizable system, so that the standard
way to prove that the system is optimizable and to develop a Riccati type theory (see, for instance, Curtain and Zwart
\cite[Section 6.2]{curtain1995introduction}) does not apply. We will instead establish directly the optimizability by ad-hoc arguments and use the LQR theory for strongly stable systems developed in Curtain and Oostveen \cite{curtain1999necessary}, see also Bensoussan, Prato, Delfour and Mitter \cite{p2} or Oostveen, Curtain and Ito \cite{p21}.
For the reader's convenience, in this section we provide, following the references above, the necessary background on infinite horizon LQR problems, with focus on systems which are strongly stable.

 Let $X$, $U$ and $Y$ be Hilbert spaces. We consider the operators $A: \mathcal{D}(A)\to X$,  which generates a $C^0$ semigroup $\mathbb{T}$ on $X$, $B\in \mathcal{L}(U,X)$ and $C\in \mathcal{L}(X,Y)$. We consider the linear system $\Sigma(A,B,C)$, of state space $X$, input space $U$ and output space $Y$, described by the equations
 \begin{align}
 	&\label{semigroupeq1}\dot{z}(t)=Az(t)+Bu(t) &(&t>0),\\
 	&\label{semigroupeq2}z(0)=z_0,\\
 	&\label{semigroupeq3}y(t)=Cz(t) &(& t>0),
 	 \end{align}
 with $z_0$ arbitrarily chosen in $X$. The infinite horizon LQR problem for \eqref{semigroupeq1}-\eqref{semigroupeq3},
 consists in minimizing the cost functional
\begin{equation}\label{costfunction}
	J(u)=\int_{0}^{\infty}\left(\left\Vert Cz(t)\right\Vert^2_Y+\left\Vert u(t)\right\Vert^2_U\right)\, {\rm d}t
\qquad\qquad (u\in L^2\left((0,\infty);\; U)\right),
\end{equation}
 over $L^2((0,\infty);\;U)$. The above problem makes sense for every $z_0\in X$ if and only if the system is {\em optimizable}, i.e., if
 for every $z_0\in X$ there exists $\overline u\in L^2\left((0,\infty);\;U\right)$ such that $J(\overline u)<\infty$.
 According to a classical result, see for instance, \cite[Theorem 6.2.4]{curtain1995introduction}, under the appropriate assumptions (including optimizability) the considered LQR problem admits a unique solution $u^*\in L^2((0,\infty);\;U)$ and there exists a self-adjoint, nonnegative operator $\Pi\in \mathcal{L}(X)$ such that \ \ $J(u^*)=\left\langle z_0,\Pi z_0\right\rangle_X$. Under supplementary assumptions (which imply, in particular, that $(A,B)$ is exponentially stabilizable), the operator $\Pi$ can be derived from the solution of an algebraic Riccati equation (see \cite[Theorem 6.2.7]{curtain1995introduction}). As already mentioned, exponential stabilizability does not hold for the system we consider in the forthcoming sections. Therefore, in order to apply a Riccati type theory, we use the approach developed in \cite{p2},
  \cite{curtain1999necessary} and  \cite{p21}, for systems which are only strongly stable. To state the result, we need the following definition.

 \begin{definition}
 For every $t>0$, the output map $\Psi_t:$ ${X}\to L^2((0,\infty); Y)$ is defined by
 \begin{equation*}
 			(\Psi_t z_0)(\sigma)=\begin{cases}
 C\mathbb{T}(\sigma)z_0 & \qquad\qquad(\sigma\in (0,t],\ z_0\in X),\\
 0                       & \qquad\qquad(\sigma > t,\ z_0\in X).	
 \end{cases}
 \end{equation*}
 The system \eqref{semigroupeq1}-\eqref{semigroupeq3} is called output stable if there exists $\Psi\in \mathcal{L}\left({X},\;L^2((0,\,\infty); \;Y)\right)$ such that
 $$
 \lim_{t\to \infty}\|\Psi_t-\Psi\|_{\mathcal{L}\left(X,\;L^2((0,\,\infty); \;Y)\right)}=0.
 $$
 	\end{definition}

 \begin{remark}\label{book_remark}
 According to the terminology introduced in Tucsnak and Weiss {\rm \cite[Section 4.6]{tucsnak2009observation}}, if operators $A$ and $C$ are as in the above definition, we say that $C$ is an infinite-time admissible observation operator for $\mathbb{T}$ and we design $\Psi$ as the extended output map.
 \end{remark}

An essential tool in proving our second main result is the proposition below, for which we refer again to \cite{curtain1999necessary}.

 \begin{proposition}\label{mainusedtheorem}
 	Suppose that there exists an operator $F\in \mathcal{L}(X, U)$ such that the operator $A+BF$ generates a strongly stable semigroup $\mathbb{T}^F$ on $X$ and with the system $\Sigma\left(A+BF,\,0,\, \begin{bmatrix} C\\F\end{bmatrix}\right)$ (with state space $X$, input space $U$ and  output space $Y\times U$) output stable. Then there exist a unique optimal control $\bar{u}\in L^2((0,\infty);\;U)$ minimizing \eqref{costfunction}.  Moreover $\bar{u}$ is given by $\bar{u}(t)=-B^*P_0 z(t)$, where $P_0\in \mathcal{L}(X)$ is the minimal self-adjoint nonnegative solution to the algebraic Riccati equation
 	\begin{equation}
 		{A}^*P z+P{A}z-P{B}{B}^*Pz+{C}^*{C} z = 0 \qquad\qquad (z\in \mathcal{D}(A)).\notag
 	\end{equation}
 Finally, the corresponding minimal cost functional is given by
 \begin{equation}
 		J(\bar{u},\,z_0)=\left<P_0 z_0,\, z_0\right>_{X},\notag
 \end{equation}
 where $\left<\cdot,\, \cdot\right>_{X}$ is the inner product in the Hilbert space ${X}$.
 \end{proposition}

 \begin{remark}\label{lab_rem}
The sense which $P$ satisfies the Riccati equation above is described in {\rm \cite[Part IV 3.1, p.398 and Part V, p.481]{p2}}, where it is shown, in particular, that $Pz\in \mathcal{D}(A^*)$ for every $z\in \mathcal{D}(A)$.
\end{remark}

\section{Operator form of the governing equations and semigroup generator}\label{section2}

In this section, we show that system \eqref{A1}-\eqref{A_Last} can be written in the standard form
\begin{align}
&\label{abstract_form}\dot z(t)=\mathcal{A}z(t)+\mathcal{B}u(t)\qquad (t>0),\\
&\label{abstract_form_2}z(0)=z_0,
\end{align}
where $\mathcal{A}$ is the generator of an analytic semigroup on a Hilbert space $\mathcal{X}$ and $\mathcal{B}\in \mathcal{L}(\mathcal{U},\mathcal{X})$,
where $\mathcal{U}$  is another Hilbert space.

To write system \eqref{A1}-\eqref{A_Last} in the form \eqref{abstract_form}-\eqref{abstract_form_2} we follow the procedure proposed in \cite{p19}, where the corresponding system, with the fluid confined in a bounded container, has been investigated.

We begin by setting
\begin{equation}\label{q_pm}
	q_-(t)=q(t,-a),\qquad q_+(t)=q(t,a) \qquad\qquad (t>0).
\end{equation}
  Next we remark that the methodology developed  in \cite{p19} to eliminate the pressure term $p$ can be applied with no major modification to our case. In this way we obtain the following system, equivalent to the original governing equations \eqref{A1}-\eqref{A_Last}:
\begin{align}
	&\label{A3}\dot{H}(t)=-\frac{q_+(t)-q_-(t)}{2a} &(& t>0),\\
	&\label{conversationofmass}\frac{\partial h}{\partial t}=-\frac{\partial q}{\partial x} &(&t>0,\; x\in \mathcal{E}),\\
	&\label{conversationofmonmutem}\frac{\partial q}{\partial t}=-\frac{\partial h}{\partial x}+\mu \frac{\partial^2 q}{\partial x^2} &(&t>0,\; x\in \mathcal{E}),\\
	&\label{boundarycondition}\left[
	\begin{gathered}
		\dot{q}_-(t)\\
		\dot{q}_+(t)
	\end{gathered}
	\right]=2aM\left[\begin{gathered}
		u(t)\\
		-u(t)
	\end{gathered}\right]\notag\\
	&+4a^2M\left[
	\begin{gathered}
		\mu\left(\frac{q_+(t)-q_-(t)}{2a}\right)+\left(h(t,-a^-)-H(t)\right)-\mu\frac{\partial q}{\partial x}(t,-a^-)\\
		-\mu\left(\frac{q_+(t)-q_-(t)}{2a}\right)-\left(h(t,a^+)-H(t)\right)+\mu\frac{\partial q}{\partial x}\left(t,a^+\right)
	\end{gathered}
	\right] &(& t>0),\\
	&\lim_{\vert x\vert \rightarrow \infty}q(t,x)=0 &(&t>0),\\
	&H(0)=H_0,\\
	&h(0,x)=h_0(x) &(&x\in \mathcal{E}),\\
	&\label{initial_date_of_q}q(0,x)=q_0(x) &(&x\in\mathcal{E}),\\
	&\label{A4} q_-(0)=q_{-,0},\qquad q_+(0)=q_{+,0},
\end{align}
where the matrix $M$ is defined by
\begin{align}
	M=\frac{1}{8a^3\left(1+\frac{2a^3}{3}\right)}\begin{bmatrix}
		1+\frac{8a^3}{3} && 1-\frac{4a^3}{3}\\
		1-\frac{4a^3}{3}  && 1+\frac{8a^3}{3}
	\end{bmatrix}.\label{AM}
\end{align}
\indent  At this stage  we introduce more notation, which will be used in all the remaining part of this paper. Firstly, we define the spaces
\begin{align}
	&\label{defineX}\mathcal{X} = \left\{ [H \,\;h\,\;q\,\;q_-\,\; q_+]^\top \in \mathbb{C} \times H^{1}(\mathcal{E})
	\times L^2(\mathcal{E})\times \mathbb{C} \times \mathbb{C}  \right\},\\
	&\label{defineZ}\mathcal{Z} = \left\{[H\,\;h\,\;q\,\;q_-\,\; q_+]^\top \in \mathbb{C}\times H^{1}(\mathcal{E}) \times H^{2}(\mathcal{E})\times\mathbb{C}\times \mathbb{C} \;\vert\; q(-a)=q_-,\, q(a)=q_+ \right\},\\
&\label{defineW}
	 	\mathcal{W}= \left\{[H\,\;h\,\;q\,\;q_-\,\; q_+]^\top \in \mathbb{C}\times H^{1}(\mathcal{E}) \times H^{1}(\mathcal{E})\times\mathbb{C}\times \mathbb{C} \;\vert \; q(-a)=q_-,\, q(a)=q_+ \right\},\\
	&\mathcal{U}=\mathbb{C}.\label{defineU}
\end{align}
Moreover, denoting by $z=[H\,\;h\,\;q\,\;q_-\,\; q_+]^\top$ a generic element of $\mathcal{X}$, we also introduce the operators $\mathcal{A}:$ $\mathcal{D}(\mathcal{A})\rightarrow {\mathcal{X}}$ and $\mathcal{B}\in \mathcal{L}(\mathcal{U}, {\mathcal{X}})$ defined by
\begin{equation}\label{defineDA}
	\mathcal{D}(\mathcal{A})=\mathcal{Z},
\end{equation}
\begin{align}
	\mathcal{A}z=\left[
	\begin{gathered}
		-\frac{q_+-q_-}{2a}\\
		-\frac{{\rm d}q}{{\rm d}x}\\
		-\frac{{\rm d}h}{{\rm d}x}+\mu\frac{{\rm d}^2q}{{\rm d}x^2}\\
		R_1z\\
		R_2z
	\end{gathered}\right], \qquad\qquad
	 \label{A12}\mathcal{B}u=\left[\begin{gathered}
		0\\
		0\\
		0\\
		\frac{au}{(1+\frac{2a^3}{3})}\\
		-\frac{au}{(1+\frac{2a^3}{3})}
	\end{gathered}
	\right]\qquad (z\in \mathcal{D}(\mathcal{A}),\ u\in \mathcal{U}),
\end{align}
where
\begin{align}\label{addes_label}
	\left[\begin{gathered}
		R_1z\\
		R_2z
	\end{gathered}\right]=4a^2M\left[\begin{gathered}
		\mu\left(\frac{q_+-q_-}{2a}\right)+	\left(h(-a^-)-H\right)-\mu\frac{{\rm d} q}{{\rm d} x}(-a^-)\\
		-\mu\left(\frac{q_+-q_-}{2a}\right)	-\left(h(a^+)-H\right)+\mu\frac{{\rm d} q}{{\rm d} x}\left(a^+\right)
	\end{gathered}\right]. \end{align}
It is easy to check that $\mathcal{B}\in \mathcal{L}(\mathcal{U},\mathcal{X})$.

With the above notation, following step by step the procedure in \cite{p19}, we have:

\begin{proposition}\label{echivalenta_mare}
Let $T>0$ and let $\left[\begin{gathered}
	H_0\ G_0\ h_0\ q_0
\end{gathered}\right]^\top$
satisfies the assumptions in Theorem \ref{th_exist_1}.
Let $\left[\begin{gathered}
 	H\
 	h\
 	q\
 	p
 \end{gathered}\right]^\top $ be a solution of  \eqref{A1}-\eqref{A_Last} satisfying
\eqref{the_regularity}-\eqref{the_regularity3}. Then $z = [H\,\;h\,\;q\,\;q_-\,\; q_+]^\top$, with $q_+$ and $q_-$ defined in \eqref{q_pm},
lies in
\begin{equation}\label{reg_W_bis}
C([0,T];\;\mathcal{W})\cap L^2((0,T);\;\mathcal{Z})\cap H^1((0,T);\;\mathcal{X}),
\end{equation}
and it satisfies
\begin{align}
	&\label{E1}\dot{z}(t)=\mathcal{A}z(t)+\mathcal{B}u(t) \qquad\qquad ( t>0) ,\\
	&\label{E3}z(0)=z_0.
\end{align}
Conversely, assume that $z$ satisfies \eqref{reg_W_bis}-\eqref{E3}. Then
there exists an extensions of $q$ (still denoted by $q$)
and $p$ such that $\left[\begin{gathered}
 	H\
 	h\
 	q\
 	p
 \end{gathered}\right]^\top $ is a solution of  \eqref{A1}-\eqref{A_Last} satisfying
\eqref{the_regularity}-\eqref{the_regularity3}.

\end{proposition}


 In this section and in the following one we prove that operator $\mathcal{A}$ generates an analytic semigroup on $\mathcal{X}$
and we give the first consequences of this result on the wellposedness of system \eqref{A3}-\eqref{A4}.
The remaining part of this section is devoted to the study of the spectrum and of the resolvents of the
 operator $\mathcal{A}$ defined above.

\indent As usual, we denote by $\mathcal{\sigma} (\mathcal{A})$ and $\mathcal{\rho} (\mathcal{A})$ the spectrum and the resolvent set of operator $\mathcal{A}$, respectively. We also set
\begin{equation}\label{defineCnegetive}
	\mathbb{C}_{-} = \left\{{\lambda\in \mathbb{C}\;\, \vert \;\, {\rm Re}\, \lambda < 0 }\right\}, \qquad \mathbb{R}_-=\mathbb{C}_-\cap \mathbb{R}.
\end{equation}

In order to state the main theorem of this section, we need four technical lemmas.
Although based on elementary facts, the proofs of these four lemmas require a certain computational effort, so that they are postponed to the appendix in subsection \ref{AppendixA}.
The first one of these lemmas is:

\begin{lemma}\label{le1}
	Assume that $\mu>0$ and $\lambda \in \mathbb{C}\verb|\| \{0\}$, $\lambda \neq -\frac{1}{\mu}$. Then $\frac{\lambda^2}{1+\mu\lambda}\in \mathbb{R}_-$ if and only if
	\begin{equation}
		\lambda \in \left(-\infty, -\frac{1}{\mu}\right) \quad {\rm or} \quad \left\vert \lambda+\frac{1}{\mu}\right\vert=\frac{1}{\mu} \ .
	\end{equation}
\end{lemma}

The second auxiliary result to be proved in the Appendix A is:

\begin{lemma}\label{le2}
	Let $a>0$ and $\mu>0$. We assume that $\lambda\in \mathbb{C}$ satisfies
	\begin{equation}
		\lambda \not\in \left(-\infty, -\frac{1}{\mu}\right] \quad {\rm and}\quad \left\vert \lambda+\frac{1}{\mu}\right\vert \neq \frac{1}{\mu}\ .\label{C1}
	\end{equation}
	Let
	\begin{equation}
		\omega_\lambda:=\sqrt{\frac{\lambda^2}{1+\mu\lambda}}\, ,\label{DEOMGA}
	\end{equation}
	where $\sqrt{\cdot}$ stands for the principal determination of the square root function. (Note that this definition
is correct due to  Lemma \ref{le1}.)

Let $\mathcal{M}_\lambda$ be defined by
	\begin{align}
		\mathcal{M}_\lambda= \begin{bmatrix}
			\lambda\left(1+\frac{8a^3}{3}\right)+2a\left(\mu +\frac{1}{\lambda}\right)+\frac{4a^2\lambda}{\omega_\lambda}&\ & -\lambda \left(1-\frac{4a^3}{3}\right)-2a\left(\mu +\frac{1}{\lambda}\right)\\
			-\lambda \left(1-\frac{4a^3}{3}\right)-2a\left(\mu +\frac{1}{\lambda}\right) &\ & \lambda\left(1+\frac{8a^3}{3}\right)+2a\left(\mu +\frac{1}{\lambda}\right)+\frac{4a^2\lambda}{\omega_\lambda}
		\end{bmatrix}.		\label{DMLAMBDA}
	\end{align}
	Then, there exists a set $\mathcal{S}\subset\mathbb{C}_-$, having at most four elements, such that for every $\lambda$ satisfying \eqref{C1} and not lying in $\mathcal{S}$ we have that the matrix $\mathcal{M}_\lambda$ is invertible.
\end{lemma}

An important role in proving our resolvent estimates is played by the following result:

\begin{lemma}\label{le3}
	Let $\omega\in \mathbb{C}$, ${\rm Re}\, \omega>0$ and $a>0$. We define the following operators:
	\begin{align}
		&\left(D_-(\omega)\gamma _-\right)(x)=\gamma_- \exp[\omega(a+x)] &(&\gamma_-\in \mathbb{C},\quad x\leq -a)\label{C4},\\
		&\label{defintion_of_operator_D_plus}\left(D_+(\omega)\gamma _+\right)(x)=\gamma_+ \exp[\omega(a-x)] &(&\gamma_+\in \mathbb{C},\quad x\geq a),
	\end{align}
	\begin{align}
		&\left(R_-(\omega)\varphi\right)(x)\notag\\
		=&\frac{1}{2\omega}\left(\int_{x}^{-a}\exp(-\omega \xi)\varphi (\xi) {\rm d}\xi\right)\exp(\omega x)+\frac{1}{2\omega}\left(\int_{-\infty}^{x}\exp(\omega \xi)\varphi(\xi) {\rm d}\xi\right)\exp(-\omega x) \nonumber\\
		&-\frac{1}{2\omega}\left(\int_{-\infty}^{-a}\exp(\omega \xi)\varphi(\xi) {\rm d}\xi\right)\exp(2\omega a)\exp(\omega x)\qquad\quad \left(\varphi\in L^2(-\infty, -a), \;x\leq -a\right),\\
		&\left(R_+(\omega)\varphi\right)(x)\notag\\
		=&\frac{1}{2\omega}\left(\int_{a}^{x}\exp(\omega \xi)\varphi (\xi) {\rm d}\xi\right)\exp(-\omega x)+\frac{1}{2\omega}\left(\int_{x}^{\infty}\exp(-\omega \xi)\varphi(\xi) {\rm d}\xi\right)\exp(\omega x) \nonumber\\
		&\label{defintion_of_operator_R_plus} -\frac{1}{2\omega}\left(\int_{a}^{\infty}\exp(-\omega \xi)\varphi(\xi) {\rm d}\xi\right)\exp(2\omega a)\exp(-\omega x)\qquad\quad \left(\varphi\in L^2(a, \infty),\; x\geq a\right).
	\end{align}
	Then we have
	\begin{align}
		&D_-(\omega)\in \mathcal{L}\left(\mathbb{C},\;L^2(-\infty,-a)\right), &D&_+(\omega)\in \mathcal{L}\left(\mathbb{C},\;L^2(a,\infty)\right),\notag\\
		&R_-(\omega)\in \mathcal{L}\left(L^2(-\infty,-a)\right),
		&R&_+(\omega)\in \mathcal{L}\left(L^2(a,\infty)\right),\notag
	\end{align}
	with the norm estimates
	\begin{align}
		&\Vert D_-(\omega)\Vert_{\mathcal{L}\left(\mathbb{C},\;L^2(-\infty, -a)\right)}\leq \left(\frac{1}{2{\rm Re}\,\omega}\right)^{\frac{1}{2}},\quad &\Vert& D_+(\omega)\Vert_{\mathcal{L}\left(\mathbb{C},\;L^2(a,\infty)\right)}\leq \left(\frac{1}{2{\rm Re}\,\omega}\right)^{\frac{1}{2}},\label{NBD}\\
		&	\Vert R_-(\omega)\Vert_{\mathcal{L}(L^2(-\infty, -a))}\leq \frac{3}{2\vert \omega\vert {\rm Re}\,\omega},\quad &\Vert& R_+(\omega)\Vert_{\mathcal{L}(L^2(a,\infty))}\leq \frac{3}{2\vert \omega\vert {\rm Re}\, \omega}.\label{NRB}
	\end{align}
\end{lemma}

Finally, the structure of the resolvent of our semigroup generator will be show to be a consequence of the following lemma:

\begin{lemma}\label{le4}
	Let $a>0$ and $\mu >0$ and assume that $\lambda\in \mathbb{C}$ satisfies condition \eqref{C1}. Then for every $\varphi\in L^2(\mathcal{E})$ and every $\gamma_-$, $\gamma_+\in \mathbb{C}$, there exists a unique $q_\lambda \in H^2(\mathcal{E})$ such that
	\begin{align}
		\left\{
		\begin{aligned}
			-&\frac{{\rm d}^2 q_\lambda}{{\rm d}x^2}+\omega^2_\lambda q_\lambda=\varphi(x)\qquad\qquad (x\in\mathcal{E}),\\
			&q_\lambda (-a)=\gamma_-, \quad q_\lambda (a)=\gamma_+,
		\end{aligned}\right.\label{Eqqlad}
	\end{align}
	where $\omega_\lambda$ is defined in \eqref{DEOMGA}. Moreover, we have
	\begin{align}\label{whole_whale}
		q_\lambda(x)=\left\{
		\begin{aligned}
			&[D_-(\omega_\lambda)\gamma_-](x)+[R_-(\omega_\lambda)\varphi](x)&(&x\leq-a),\\\\
			&[D_+(\omega_\lambda)\gamma_+](x)+[R_+(\omega_\lambda)\varphi](x)&(&x\geq a),
		\end{aligned}\right.
	\end{align}
	where the operators $D_+(\omega_\lambda)$, $D_-(\omega_\lambda)$, $R_+(\omega_\lambda)$ and $R_-(\omega_\lambda)$ have been defined in Lemma \ref{le3}.	
\end{lemma}

The main result in this section is:

\begin{theorem}\label{theo 2}
	For every $a>0$ and $\mu >0$, there exists a set $\mathcal{S}\subset \mathbb{C}_-$, having at most four elements, such that the spectrum of the operator $\mathcal{A}$ defined in \eqref{defineX}-\eqref{addes_label} satisfies
	\begin{equation}\label{sperA}
		\sigma(\mathcal{A})\subset E,
	\end{equation}
	where $E$ is defined by
	\begin{equation}
		E=\{0\}\cup \mathcal{S}\cup \left(-\infty, -\frac{1}{\mu}\right]\cup \left\{\lambda \in \mathbb{C}_-\quad \left\vert \quad \left\vert \lambda +\frac{1}{\mu}\right \vert =\frac{1}{\mu}\right\}\right..\notag
	\end{equation}
   Moreover,  for every $\lambda\not\in E$, $\mathcal{F}=[f_1\;\, f_2\;\, f_3\;\, f_4 \;\, f_5]^\top \in \mathcal{X}$  we denote
	\begin{align}
		&\label{resovlent_equation}(\lambda \mathbb{I}-\mathcal{\mathcal{A}})^{-1}\mathcal{F}=z_\lambda=[H_\lambda\;\,   h_\lambda\;\, q_\lambda \;\, q _{-,\lambda}\;\,  q _{+,\lambda}]^\top \in \mathcal{D}(\mathcal{A}).
	\end{align}
Then we have
	\begin{align}\label{REeq}
		&\left[\begin{gathered}
			q_{-,\lambda}\\
			q_{+,\lambda}
		\end{gathered}\right]\notag\\
		=&\mathcal{M}^{-1}_\lambda \left(M^{-1} \left[\begin{gathered}
		f_{4}\\
		f_{5}
		\end{gathered}\right]+\frac{4a^2}{\lambda}\left[\begin{gathered}
		f_{2}(-a^-)-f_1\\
		f_{1}-f_{2}\left(a^+\right)
		\end{gathered}\right]+\frac{4a^2\lambda}{\omega^2_\lambda}\left[\begin{gathered}
		-\int_{-\infty}^{0}\exp(\omega_\lambda\xi)\varphi_\lambda (\xi-a){\rm d}\xi\\
		\int_{0}^{\infty}\exp(-\omega_\lambda\xi)\varphi_\lambda (\xi+a){\rm d}\xi
		\end{gathered}\right]\right),
	\end{align}
where\begin{equation}
	\varphi_\lambda(x)=\frac{\lambda}{1+\mu\lambda}f_3(x)-\frac{1}{1+\mu\lambda}\frac{{\rm d}f_2}{{\rm d}x}(x)\qquad\qquad (x\in \mathcal{E}),\notag
\end{equation}
 and $\omega_\lambda$, $\mathcal{M}_\lambda$ have been defined in \eqref{DEOMGA} and \eqref{DMLAMBDA}, respectively. In addition, we have
\begin{align}
	&q_\lambda(x)=\left\{\begin{aligned}
	&[D_-(\omega_\lambda)q_{-,\lambda}](x)+[R_-(\omega_\lambda)\varphi_{\lambda}](x) \quad\;\, &(&x\leq-a),\\
	&[D_+(\omega_\lambda)q_{+,\lambda}](x)+[R_+(\omega_\lambda)\varphi_{\lambda}](x)  &(&x\geq a),
	\end{aligned}
	\right. \label{Expq}\\
	&h_\lambda(x)=\frac{1}{\lambda}\left(f_2(x)-\frac{{\rm d}	q_\lambda(x)}{{\rm d}x}\right) \qquad\qquad\qquad\qquad\;\;\, (x\in \mathcal{E}),\label{Exph}\\
	&H_\lambda=\frac{1}{\lambda}\left(f_1-\frac{q_{+,\lambda}-q_{-,\lambda}}{2a}\right),\label{ExpH}
\end{align}
where the operators $D_-(\omega_\lambda)$, $D_+(\omega_\lambda)$, $R_-(\omega_\lambda)$ and $R_+(\omega_\lambda)$ have been defined in Lemma \ref{le3}.
\end{theorem}
\begin{proof}
	We firstly remark that the resolvent equation \eqref{resovlent_equation} can be written as:
	\begin{align}
		&\label{R1}\lambda H_\lambda+\frac{q_{+,\lambda}-q_{-,\lambda}}{2a}=f_1,\\
		&\label{R2}\lambda h_\lambda(x)+\frac{{\rm d}q_\lambda}{{\rm d}x}(x)=f_2(x)  \qquad\qquad\qquad\quad\;(x\in \mathcal{E}),\\
		&\label{R3}\lambda q_\lambda(x)+\frac{{\rm d}h_\lambda}{{\rm d}x}(x)-\mu \frac{{\rm d}^2 q_\lambda}{{\rm d}x^2}(x)=f_3(x) \qquad (x\in \mathcal{E}),\\
		&\label{R4} q_\lambda(-a)=q_{-,\lambda},\qquad\quad q_\lambda(a)=q_{+,\lambda},\\
		&\label{R5}\lambda \left[\begin{gathered}
			q_{-,\lambda}\\
			q_{+,\lambda}
		\end{gathered}\right]-4a^2M\left[\begin{gathered}
		\mu\left(\frac{q_{+,\lambda}-q_{-,\lambda}}{2a}\right)+\left(h_\lambda(-a^-)-H_\lambda\right)-\mu\frac{{\rm d}q_\lambda}{{\rm d} x}(-a^-)\\
		-\mu\left(\frac{q_{+,\lambda}-q_{-,\lambda}}{2a}\right)
		-\left(h_\lambda\left(a^+\right)-H_\lambda\right)+\mu\frac{{\rm d} q_\lambda}{{\rm d} x}\left(a^+\right)
		\end{gathered}	\right]=\left[\begin{gathered}
		f_4\\
		f_5
		\end{gathered}\right],
	\end{align}
	where the matrix $M$ has been defined in \eqref{AM}.
	 From equations \eqref{R2}-\eqref{R4}, it follows that
	\begin{align}
		\left\{
		\begin{aligned}
			-&\frac{{\rm d}^2 q_\lambda}{{\rm d}x^2}(x)+\omega^2_\lambda q_\lambda(x)=\frac{\lambda}{1+\mu\lambda}f_3(x)-\frac{1}{1+\mu\lambda}\frac{{\rm d}f_2}{{\rm d}x}(x)\quad (x\in\mathcal{E}),\\
			&q_\lambda (-a)=q_{-,\lambda}, \quad q_\lambda (a)=q_{+,\lambda},\\
		\end{aligned}\right.\label{Eqq}
	\end{align}
This implies, using Lemma \ref{le4}, that \eqref{Expq} holds true. Formulas \eqref{Exph} and \eqref{ExpH}
are then obtained as direct consequences of \eqref{R1} and \eqref{R3}.
	
We still have to prove \eqref{REeq}. To this aim, we  use Lemma \ref{le4} in order to reduce equations \eqref{R5} to a linear algebraic system of unknowns $q_{-,\lambda}$ and $q_{+,\lambda}$. We note that from the first equation in \eqref{Expq}, it follows that for every $x\geqslant a$, we have
	\begin{align}\label{Eedq}
		\frac{{\rm d}q_\lambda}{{\rm d}x}(x)=&-\frac{1}{2}\left(\int_{a}^{x}\exp(\omega_\lambda\xi)\varphi_\lambda(\xi){\rm d}\xi\right)\exp(-\omega_\lambda x)+\frac{1}{2}\left(\int_{x}^{\infty}\exp(-\omega_\lambda\xi)\varphi_\lambda(\xi){\rm d}\xi\right)\exp(\omega_\lambda x)\notag \\
		&-\omega_\lambda q_{+,\lambda} \exp(\omega_\lambda(a-x))+\frac{1}{2}\left(\int_{a}^{\infty}\exp(-\omega_\lambda\xi)\varphi_\lambda(\xi){\rm d}\xi\right)\exp(2\omega_\lambda a)\exp(-\omega_\lambda x).
	\end{align}
Taking the limit when $x\to a^+$ in the above equation, it follows that
	\begin{equation}
		\frac{{\rm d}q_\lambda}{{\rm d}x}\left(a^+\right)=\left(\int_{a}^{\infty}\exp(-\omega_\lambda\xi)\varphi_\lambda(\xi){\rm d}\xi\right)\exp(\omega_\lambda a)-\omega_\lambda q_{+,\lambda}.
	\end{equation}
	In a similar manner, we have
	\begin{equation}
		\frac{{\rm d}q_\lambda}{{\rm d}x}(-a^-)=-\left(\int_{-\infty}^{-a}\exp(\omega_\lambda\xi)\varphi_\lambda(\xi){\rm d}\xi\right)\exp(\omega_\lambda a)+\omega_\lambda q_{-,\lambda}.
	\end{equation}
	Inserting the above two formulas in \eqref{R5} and using the fact that
	\begin{equation}
		\frac{1+\mu \lambda}{\lambda}=\frac{\lambda}{\omega^2_\lambda},\notag
	\end{equation}
	we obtain that
	\begin{align}
			&\lambda M^{-1} \left[\begin{gathered}
			q_{-,\lambda}\\
			q_{+,\lambda}
		\end{gathered}\right]	-4a^2\left[\begin{gathered}
		-\left(\frac{\mu}{2a}+\frac{\lambda}{\omega_\lambda}+\frac{1}{2a\lambda }\right)q_{-,\lambda}+
		\left(\frac{\mu}{2a}+\frac{1}{2a\lambda }
		\right)q_{+,\lambda}\\
		\left(\frac{\mu}{2a}+\frac{1}{2a\lambda }\right)q_{-,\lambda}-\left(\frac{\mu}{2a}+\frac{\lambda}{\omega_\lambda}+\frac{1}{2a\lambda}\right)q_{+,\lambda}
	\end{gathered}\right]\notag\\
	=&M^{-1} \left[\begin{gathered}
		f_{4}\\
		f_{5}
	\end{gathered}\right]+\frac{4a^2}{\lambda}\left[\begin{gathered}
	f_{2}(-a^-)-f_1\\
	f_{1}-f_{2}\left(a^+\right)
	\end{gathered}\right]+\frac{4a^2\lambda}{\omega^2_\lambda}\left[\begin{gathered}
	-\int_{-\infty}^{0}\exp(\omega_\lambda\xi)\varphi_\lambda (\xi-a){\rm d}\xi\\
	\int_{0}^{\infty}\exp(-\omega_\lambda\xi)\varphi_\lambda (\xi+a){\rm d}\xi
	\end{gathered}\right].\notag
\end{align}
The above equation can be rewritten
\begin{equation}
	\mathcal{M}_\lambda \left[\begin{gathered}
		q_{-,\lambda}\\
		q_{+,\lambda}
	\end{gathered}\right]=M^{-1} \left[\begin{gathered}
	f_{4}\\
	f_{5}
	\end{gathered}\right]+\frac{4a^2}{\lambda}\left[\begin{gathered}
	f_{2}(-a^-)-f_1\\
	f_{1}-f_{2}\left(a^+\right)
	\end{gathered}\right]+\frac{4a^2\lambda}{\omega^2_\lambda}\left[\begin{gathered}
	-\int_{-\infty}^{0}\exp(\omega_\lambda\xi)\varphi_\lambda (\xi-a){\rm d}\xi\\
	\int_{0}^{\infty}\exp(-\omega_\lambda\xi)\varphi_\lambda (\xi+a){\rm d}\xi
	\end{gathered}\right],\notag
\end{equation}
where the matrices $M$ and $\mathcal{M}_\lambda$ have been defined in\eqref{AM}, \eqref{DMLAMBDA}, respectively.
By combining the above formula and Lemma \ref{le2}, we obtain \eqref{REeq} and thus the
conclusion of Theorem \ref{theo 2}.
\end{proof}

\section{Semigroup generation and well-posedness}\label{sec_well_new}

In this section we continue to use the notation introduced in the previous one, in particular for the space $\mathcal{X}$ defined in \eqref{defineX} and for the operator $\mathcal{A}$ defined in \eqref{defineDA}-\eqref{addes_label}. We prove below that $\mathcal{A}$ generates an analytic semigroup on $\mathcal{X}$ and we give the proof of Theorem \ref{th_exist_1}.

To proof of the fact that $\mathcal{A}$ generates an analytic semigroup on $\mathcal{X}$ is partially based on two technical results given below. Although based on elementary facts, the proofs of these two lemmas are provided, for completeness, in the Appendix on subsection \ref{Appendix B}. The first of these lemmas is:

\begin{lemma}\label{le5}
	Let $\mu>0$ and $\theta \in [0,\frac{\pi}{2})$. Then
	\begin{equation}
		\frac{\lambda^2}{1+\mu\lambda}\not\in \mathbb{R}_- \qquad\qquad \left(\lambda\in \Sigma_\theta,\; |\lambda|\geq \frac{4}{\mu(1-\sin (\theta))}\right), \label{F1}
	\end{equation}
		and
	\begin{equation}
	{\rm Re}\,\left(\sqrt{\frac{\lambda^2}{1+\mu\lambda}}\right)\geq \frac{1}{4}\sqrt{\frac{|\lambda|(1-\sin(\theta))}{\mu}} \qquad\qquad \left(\lambda\in \Sigma_\theta,\; |\lambda|\geq \frac{4}{\mu(1-\sin (\theta))}\right),\label{F2}
	\end{equation}
	where
\begin{equation}\label{SIGMA_THETA}
\Sigma_\theta=\left\{{\lambda=\rho\exp(i\phi)\in \mathbb{C}\;\, \Big\vert \;\, \arg\phi\in\left(-\frac{\pi}{2}-\theta,\; \frac{\pi}{2}+\theta\right)}\right\} .
\end{equation}
\end{lemma}

The second technical result, which is also proved in the Appendix on subsection \ref{Appendix B} is:

\begin{lemma}\label{le6}
	Let $a,\,\mu>0$  and $\theta\in\left[0,\frac{\pi}{2}\right)$. Then there exists
 a constant $R(a,\, \mu,\,\theta)$, such that for every
	\begin{equation}
		\lambda \in \Sigma_\theta,\quad \vert\lambda\vert\geq R(a,\,\mu,\,\theta),\label{F6}
	\end{equation}
	the matrix $\mathcal{M}_\lambda$ defined in \eqref{DMLAMBDA} is invertible and there is a constant $K(\theta,\, \mu)>0$, such that
	\begin{equation}
		\left\Vert\lambda\mathcal{M}^{-1}_\lambda\right\Vert\leq K(\theta,\mu)\qquad \left(	\lambda \in \Sigma_\theta,\; \vert\lambda\vert\geq R(a,\,\mu,\,\theta)\right).\label{F7}
	\end{equation}
\end{lemma}

An important ingredient of the proof of the fact that $\mathcal{A}$ generates an analytic semigroup on $\mathcal{X}$ is the following result:

\begin{proposition}\label{rem1}
	With the notation in Theorem \ref{theo 2}, for every $\lambda\in \Sigma_\theta$, $|\lambda|\geq\frac{4}{\mu(1-\sin\theta)}$ we have
\begin{align}\label{ql}
		\left\Vert q_\lambda\right\Vert_{L^2(\mathcal{E})}\leq \tilde{c}\left(\frac{1}{|\omega_\lambda|{\rm Re}\,\omega_\lambda}\Vert\varphi_\lambda\Vert_{L^2(\mathcal{E})}+\frac{1}{({\rm Re}\, \omega_\lambda)^\frac12}\left\Vert\left[
			\begin{gathered}
				q_{-,\lambda}\\
				q_{+,\lambda}
			\end{gathered}\right]\right\Vert_{\mathbb{C}^2}\right),
	\end{align}
and
\begin{align}\label{dql}
		\left\Vert \frac{{\rm d}q_\lambda}{{\rm d}x}\right\Vert_{L^2(\mathcal{E})}\leq \tilde{c}\left(\frac{1}{{\rm Re}\,\omega_\lambda}\Vert\varphi_\lambda\Vert_{L^2(\mathcal{E})}+\left\Vert\omega_\lambda\left[
			\begin{gathered}
				q_{-,\lambda}\\
				q_{+,\lambda}
			\end{gathered}\right]\right\Vert_{\mathbb{C}^2}\right),
	\end{align}
where $\tilde{c}>0$ is a constant depending only on $a$, $\theta$ and $\mu$.
\end{proposition}

\begin{proof}
 Firstly, by combining \eqref{Expq} and Lemma \ref{le3}, we have
 \begin{equation*}
 	\left\Vert q_\lambda \right\Vert_{L^2(\mathcal{E})}\leq \tilde{c}\left(\frac{1}{|\omega_\lambda|{\rm Re}\,\omega_\lambda}\Vert\varphi_\lambda\Vert_{L^2(\mathcal{E})}+\frac{1}{({\rm Re}\, \omega_\lambda)^\frac12}\left\Vert\left[
 	\begin{gathered}
 		q_{-,\lambda}\\
 		q_{+,\lambda}
 	\end{gathered}\right]\right\Vert_{\mathbb{C}^2}\right) \qquad \qquad (\lambda\in \sigma(\mathcal{A})).
 \end{equation*}
 Similarly, we obtain
 \begin{equation*}
 	\left\Vert \frac{{\rm d}q_\lambda}{{\rm d}x} \right\Vert_{L^2(\mathcal{E})}\leq \tilde{c}\left(\frac{1}{{\rm Re}\,\omega_\lambda}\Vert\varphi_\lambda\Vert_{L^2(\mathcal{E})}+\left\Vert\frac{\omega_\lambda}{({\rm Re}\, \omega_\lambda)^\frac12}\left[
 	\begin{gathered}
 		q_{-,\lambda}\\
 		q_{+,\lambda}
 	\end{gathered}\right]\right\Vert_{\mathbb{C}^2}\right),
 \end{equation*}
 for $\lambda\in \Sigma_\theta$, $|\lambda|\geq\frac{4}{\mu(1-\sin\theta)}$. The last estimate and the fact that
 \begin{equation*}
 	\frac{1}{({\rm Re}\, \omega_\lambda)^\frac12}\geq 1.
 \end{equation*}
 imply the conclusion \eqref{dql}.
\end{proof}

We are now in  position to prove that $\mathcal{A}$ is a sectorial operator on $\mathcal{X}$.

\begin{theorem}\label{analyticsemigroup}
	Let $a,\;\mu>0$ and $\theta\in \left[0,\frac{\pi}{2}\right)$, then there exists positive constants $c=c(a,\,\theta,\,\mu)$ and $R(a,\, \mu,\,\theta)$ such that
	\begin{equation}\label{put_together}
		\left\Vert \lambda(\lambda\mathbb{I}-\mathcal{A})^{-1}\right\Vert_{\mathcal{L}(\mathcal{X})}\leq c \qquad\qquad \left(	\lambda \in \Sigma_\theta,\; \vert\lambda\vert\geq R(a,\,\mu,\,\theta)\right),
	\end{equation}
	so that the operator $\mathcal{A}$ generates an analytic semigroup on $\mathcal{X}$.
\end{theorem}

\begin{proof}
	From Theorem \ref{theo 2}, it follows that if $R(a,\, \mu,\,\theta)$ is large enough then
	\begin{equation}
		\left\{\lambda \in \Sigma_\theta,\; \vert\lambda\vert\geq R(a,\,\mu,\,\theta)\right\}\subset \rho(\mathcal{A}).\notag
	\end{equation}
	Using again Theorem \ref{theo 2}, it follows that for $\lambda$ satisfying \eqref{F6} and every $\mathcal{F}=[f_1 \,\; f_2 \,\; f_3 \,\;f_4 \,\;f_5]^\top \in \mathcal{X}$, we can set
	\begin{equation}
		(\lambda\mathbb{I}-\mathcal{A})^{-1}\mathcal{F}=z_\lambda=\left[H_\lambda \,\;h_\lambda \,\; q_\lambda \,\; q_{-,\lambda} \,\;q_{+,\lambda}\right]^\top,\notag
	\end{equation}
	with $H_\lambda$, $h_\lambda$, $q_\lambda$, $q_{-,\lambda}$ and $q_{+,\lambda}$ satisfying \eqref{REeq}-\eqref{ExpH}. Combining this with \eqref{F7}, we get 
	\begin{align}
		\left\Vert\lambda\left[
		\begin{gathered}
			q_{-,\lambda}\\
			q_{+,\lambda}
		\end{gathered}\right]\right\Vert_{\mathbb{C}^2}\leq& \tilde{c}\left(
		\left\vert\int_{-\infty}^{0}\exp(\omega_\lambda\xi)\varphi_\lambda(\xi-a){\rm d}\xi\right\vert+\left\vert\int_{0}^{\infty}\exp(-\omega_\lambda\xi)\varphi_\lambda(\xi+a){\rm d}\xi\right\vert
		\right)\notag\\
		&+\tilde{c}\left(\vert f_1\vert+\Vert f_2\Vert_{H^1(\mathcal{E})}+\vert f_4\vert +\vert f_5\vert\right) \qquad\qquad \left(\lambda \in \Sigma_\theta,\; \vert\lambda\vert\geq R(a,\,\mu,\,\theta)\right),\notag
	\end{align}
	where $$\varphi_\lambda (x)=\frac{\lambda}{1+\mu\lambda}f_3(x)-\frac{1}{1+\mu\lambda}\frac{{\rm d}f_2(x)}{{\rm d}x}
\qquad\qquad (x\in \mathcal{E}),$$
and  $\tilde{c}=\tilde{c}(\theta,\;\mu\;,a)$.
	 Noting that
	\begin{equation}
		\Vert \varphi_\lambda\Vert_{L^2(\mathcal{E})}\leq \tilde{c}\left(\Vert f_3\Vert_{L^2(\mathcal{E})}+\frac{1}{\vert\lambda\vert}\Vert f_2\Vert_{H^1(\mathcal{E})}\right)\qquad\qquad \left(\lambda \in \Sigma_\theta,\; \vert\lambda\vert\geq R(a,\,\mu,\,\theta)\right),\notag
	\end{equation}
	the above inequalities, combined with H\"older's  inequality, imply that
	\begin{align}
		&\left\Vert\lambda\left[
		\begin{gathered}
			q_{-,\lambda}\\
			q_{+,\lambda}
		\end{gathered}\right]\right\Vert_{\mathbb{C}^2}\notag\\
		\leq&\tilde{c}\left(\vert f_1\vert+\Vert f_2\Vert_{H^1(\mathcal{E})}+\Vert f_3\Vert_{L^2(\mathcal{E})}+\vert f_4\vert +\vert f_5\vert\right)
\qquad \qquad \left(\lambda \in \Sigma_\theta,\; \vert\lambda\vert\geq R(a,\,\mu,\,\theta)\right).\label{1_form}
	\end{align}
	On the other hand, combining Lemma \ref{le3}, Lemma \ref{le5} and \eqref{Expq}, we have
	\begin{align}\label{eastimateofq}
		\Vert \lambda q_\lambda \Vert_{L^2(\mathcal{E})}
		\leq &	\tilde{c}		\left\Vert\lambda\left[
		\begin{gathered}
			q_{-,\lambda}\\
			q_{+,\lambda}
		\end{gathered}\right]\right\Vert_{\mathbb{C}^2} \notag \\
&+\frac{\vert \lambda\vert }{2\vert \omega_\lambda\vert {\rm Re}\, \omega_\lambda}\left(\Vert f_3\Vert_{L^2(\mathcal{E})}+\frac{1}{\vert \lambda \vert}\Vert f_2\Vert_{H^1(\mathcal{E})}\right) \qquad\qquad\left(\lambda \in \Sigma_\theta,\; \vert\lambda\vert\geq R(a,\,\mu,\,\theta)\right).
	\end{align}
	On the other hand, since
	\begin{equation}
		{\rm Re}\,\sqrt{z}=\frac{1}{\sqrt{2}}\sqrt{\vert z\vert + {\rm Re}\,z}\qquad\qquad (z\in\mathbb{C}\verb|\|\mathbb{R}_-),
	\end{equation}
for $\vert\lambda\vert$ large enough, we have
	\begin{equation}
	\frac{\vert \lambda\vert }{2\vert \omega_\lambda\vert {\rm Re}\, \omega_\lambda}\leq c,\label{addestimate}
	\end{equation}
	where $c$ is a constant depending only on $a$, $\mu$ and $\theta$.
	
	Estimates \eqref{eastimateofq} and \eqref{addestimate}, combined with \eqref{F2} and \eqref{F7} yield that
	\begin{align}
			&\Vert \lambda q_\lambda \Vert_{L^2(\mathcal{E})}\notag\\
		\leq&\tilde{c}\left(\vert f_1\vert+\Vert f_2\Vert_{H^1(\mathcal{E})}+\Vert f_3\Vert_{L^2(\mathcal{E})}+\vert f_4\vert +\vert f_5\vert\right) \qquad\qquad \left(\lambda \in \Sigma_\theta,\; \vert\lambda\vert\geq R(a,\,\mu,\,\theta)\right).\label{2_form}
	\end{align}
	 Using Theorem \ref{theo 2} and Remark \ref{rem1}, we derive:
	 \begin{align}
	 	&\left\Vert\lambda h_\lambda \right\Vert_{H^1(\mathcal{E})}\notag\\
	 	=&\left\Vert f_2-\frac{{\rm d}q_\lambda}{{\rm d}x}\right\Vert_{H^1(\mathcal{E})}\notag\\
	 	\leq & \Vert f_2\Vert_{H^1(\mathcal{E})}+\left\Vert \frac{{\rm d}q_\lambda}{{\rm d}x}\right\Vert_{L^2(\mathcal{E})}+\left\Vert \frac{{\rm d}^2q_\lambda}{{\rm d}x^2}\right\Vert_{L^2(\mathcal{E})}\notag\\
	 	\leq&\tilde{c}\left(\vert f_1\vert+\Vert f_2\Vert_{H^1(\mathcal{E})}+\Vert f_3\Vert_{L^2(\mathcal{E})}+\vert f_4\vert +\vert f_5\vert\right) \qquad\qquad \left(\lambda \in \Sigma_\theta,\; \vert\lambda\vert\geq R(a,\,\mu,\,\theta)\right), \label{3_form}
	 \end{align}
	and
	\begin{align}
		&\Vert\lambda H_\lambda\Vert_\mathbb{C}\notag\\
		=&\left\Vert f_1-\frac{q_{+,\lambda}-q_{-,\lambda}}{2a}\right\Vert_\mathbb{C}\notag\\
		\leq&\tilde{c}\left(\vert f_1\vert+\Vert f_2\Vert_{H^1(\mathcal{E})}+\Vert f_3\Vert_{L^2(\mathcal{E})}+\vert f_4\vert +\vert f_5\vert\right) \qquad\qquad \left(\lambda \in \Sigma_\theta,\; \vert\lambda\vert\geq R(a,\,\mu,\,\theta)\right).\label{4_form}
	\end{align}
Putting together \eqref{1_form}, \eqref{2_form}, \eqref{3_form} and \eqref{4_form} yield the conclusion \eqref{put_together}.
\end{proof}

We are now in a position to prove our first main result.
\begin{proof}[Proof of Theorem \ref{th_exist_1}]
The proof is based on Theorem \ref{analyticsemigroup} and on an interpolation argument. Following
Lunardi \cite[Chapter 1]{lunardi2018interpolation}, given two Banach spaces $X,\ Y$, where $Y\subset X$
with continuous embedding, we denote by $(X,\,Y)_{\frac12,2}$ the real interpolation space of order $\frac12$
between $Y$ and $X$.

Recall the operators  $\mathcal{A}$ and $\mathcal{B}$ defined in
	\eqref{defineDA} and \eqref{A12}. It is not difficult to check that the real interpolation space
$\left(\mathcal{X},\, \mathcal{D}(\mathcal{A})\right)_{\frac{1}{2},2}$ coincides with the space $\mathcal{W}$ defined in
\eqref{defineW}. Moreover, we know from Theorem \ref{analyticsemigroup}
that $\mathcal{A}$ generates an analytic semigroup on $\mathcal{X}$ and we also know that $\mathcal{B}\in \mathcal{L}(\mathcal{U},\mathcal{X})$.
According to a classical result on analytic semigroups (see, for instance, \cite[Theorem 2.1, p.207]{p2}), for every $z_0\in \left(\mathcal{X}, \mathcal{D}(\mathcal{A})\right)_{\frac{1}{2},2}=\mathcal{W}$, $T>0$ and $\mathcal{B}u\in L^2((0,T);\; \mathcal{X})$, the system
\eqref{abstract_form}, \eqref{abstract_form_2} admits a unique solution
$$
z\in C([0,T];\;\mathcal{W})\cap L^2((0,T);\;\mathcal{Z})\cap H^1((0,T);\;\mathcal{X}).
$$
The desired conclusion follows now by applying Proposition \ref{echivalenta_mare}.
\end{proof}

\section{The optimal control problem}\label{section4}

We give below the proof of our main result in Theorem \ref{thA1} and we describe the construction of the feedback control by applying
Proposition \ref{mainusedtheorem}. To apply this result we have to prove the existence of a feedback operator $F$
and of an observation operator $C$
such that our system satisfies all the assumptions of Proposition \ref{mainusedtheorem}. The definition of $C$
is easy to find, whereas
the construction of $F$ is slightly more involved. Therefore, in the first subsection below we
derive some energy estimates which will suggest the choice of this feedback operator.

\subsection{Feedback design}

We begin by deriving the following energy estimate:

\begin{proposition}\label{leA1} Let $\left[\begin{gathered}
 	H\
 	h\
 	q\
 	p
 \end{gathered}\right]^\top $ be the solution of system \eqref{A1}-\eqref{A_Last} whose existence and uniqueness has been proved in Theorem \ref{th_exist_1} and let
	\begin{gather}
		E(t)=\int_{\mathbb{R}} \frac{1}{2}\left[q^2(t,x)+h^2(t,x)\right]{\rm d}x+\frac{1}{2}\dot{H}^2(t) \qquad\qquad (t> 0).\label{B8}
	\end{gather}
	 Then we have
	\begin{equation}
		\frac{\rm d}{{\rm d} t}E(t)=-\mu \int_{\mathbb{R}} \left (\frac{\partial q}{\partial x}\right )^2(t,x) {\rm d}x+u(t)\dot{H} (t) \qquad\qquad (t> 0). \label{EEst}
	\end{equation}
	
\end{proposition}

\begin{proof}
	From \eqref{B8}, it follows that
	\begin{gather}
		\frac{\rm d}{{\rm d} t}E(t)=\int_{\mathbb{R}}\left[q(t,x)\frac{\partial q}{\partial t}(t,x) + h(t,x) \frac{\partial h}{\partial t}(t,x)\right] {\rm d}x + \dot{H}(t)\ddot{H}(t)\qquad\qquad (t>0\ a.e.). \nonumber
	\end{gather}
	 Combining the above formula with \eqref{A1}, \eqref{conversation_of_monmentun}, \eqref{AHh} and \eqref{p=0}, we obtain
	\begin{align}
		\frac{\rm d}{{\rm d} t}E(t)
		=&-\int_{\mathcal{E}}\frac{\partial(hq)}{\partial x}(t,x) {\rm d}x+\mu \int_{\mathcal{E}} q(t,x)\frac{\partial^2 q}{\partial x^2}(t,x) {\rm d}x\notag\\
		 &+\mu \int_{\mathcal{I}} q(t,x)\frac{\partial^2 q}{\partial x^2}(t,x) {\rm d}x - H(t)\int_{\mathcal{I}}\frac{\partial q}{\partial x}(t,x){\rm d}x\notag\\
		 &-\int_{\mathcal{I}}q(t,x)\frac{\partial p}{\partial x}(t,x){\rm d}x+\dot{H}(t)\ddot{H}(t) \qquad\qquad\qquad \quad\qquad (t>0 \ a.e.).\notag
	\end{align}
	Integrating by parts and combining  \eqref{jump_condition_1}, \eqref{jump_condition_2}, \eqref{theconnectionofHp} and  \eqref{A2}, we obtain
	\begin{equation}
		\frac{\rm d}{{\rm d} t}E(t)=-\mu \int_{\mathbb{R}} \left (\frac{\partial q}{\partial x}\right )^2(t,x) {\rm d}x-\dot{H}(t)\int_{\mathcal{I}}p(t,x){\rm d}x+\dot{H}(t)\ddot{H}(t) \qquad\qquad (t>0\ a.e.).\notag
	\end{equation}
	Using \eqref{newton_second_law} and above equality, it follows that
	\begin{equation}
		\frac{\rm d}{\rm dt}E(t)=-\mu \int_{\mathbb{R}} \left (\frac{\partial q}{\partial x}\right )^2(t,x) {\rm d}x+u(t)\dot{H}(t) \qquad\qquad (t> 0),\notag
	\end{equation}
which concludes the proof of this lemma.
\end{proof}

The above lemma suggests that $\tilde u_\alpha(t)=-\alpha \dot H(t)$, with $\alpha\geqslant 0$, is a feedback control for which
the energy of the system is not increasing. Moreover, if $\alpha>0$ it follows from \eqref{EEst} that $J(\tilde u_\alpha)$ is finite.
In order to express this feedback law in terms of $z$ introduced in Section \ref{section2}, we note that from \eqref{A3} this feedback
can be rewritten
$$
\tilde{u}_\alpha(t)=\alpha\left(\frac{q_+(t) - q_-(t)}{2a}\right)\qquad \qquad (t>0,\; \alpha\geq 0).
$$
In the next subsection we show  how the above formula allows us to construct a feedback operator $F$ and operators $A$, $B$ and $C$
such that all the assumptions in Proposition \ref{mainusedtheorem} are satisfied.

\subsection{Proof of Theorem \ref{thA1} }

In this subsection we show that equations \eqref{A1}-\eqref{A2}, together with the output law $y(t)=\dot H(t)$ for $t>0$ define a system satisfying all the assumptions in Proposition \ref{mainusedtheorem}. To this aim, we define the following spaces and operators:
\begin{itemize}
\item
The state space is $X=\mathcal{W}$, where $\mathcal{W}$ has been introduced in \eqref{defineW};
\item
$A$ is the part in $\mathcal{W}$ of the operator $\mathcal{A}$ introduced in \eqref{defineDA}-\eqref{addes_label};
\item
$Y=\mathbb{C}$ and $C\in \mathcal{L}(X,Y)$ is defined by
\begin{equation}\label{def_C_bun}
Cz=	-\frac{q_+-q_-}{2a} \qquad\qquad \left(z=[H \,\;h \,\; q \,\;q_- \,\; q_+]^\top \in X\right).
\end{equation}
\item
$U=\mathbb{C}$ and $B=\mathcal{B}\in \mathcal{L}(U, X)$, where $\mathcal{B}$ has been defined in \eqref{A12};
\item $F\in \mathcal{L}(X,U)$ is defined by
\begin{equation}\label{DEF}
F z=\frac{q_+-q_-}{2a} \qquad\qquad(z\in X).
\end{equation}
\end{itemize}

\begin{remark}\label{the_equal_system}
	The motivation behind of the above choice of the feedback operator has been explained above, in the comments following the proof of  Proposition \ref{leA1}.
\end{remark}

\begin{proposition}\label{anal_perturb}
	With the above notation, the operators $A$ and $A_F=A+BF$ generate analytic semigroups on $X$.
\end{proposition}

\begin{proof}
As already mentioned in the Proof of Theorem \ref{th_exist_1}, $X=\mathcal{W}$ coincides with the real interpolation
space  $\left(\mathcal{X}, \mathcal{D}(\mathcal{A})\right)_{\frac{1}{2},2}$, where $\mathcal{A}$
has been introduced in \eqref{defineDA}-\eqref{addes_label}. On the other hand, we have seen in Theorem \ref{analyticsemigroup}
that $\mathcal{A}$ is a sectorial operator on the space $\mathcal{X}$ defined in \eqref{defineX}.
The two above facts, combined with \cite[Theorem 6.5, p.165]{lunardi2018interpolation} imply that $A$ is the part of $\mathcal{A}$ in $X$ and $A$ is sectorial on $X$, so that $A$ generates an analytic semigroup on $X$. This, the fact that $BF\in \mathcal{L}(X)$ and Corollary 2.2
in Pazy \cite[p.81]{p23} imply that $A_F$ generates an analytic semigroup on $X$.
\end{proof}

The result below provides some useful spectral properties of the operator $A_F$ introduced above.

\begin{proposition}\label{thD1}
		With the notation in Proposition \ref{anal_perturb}, there exists a set $\mathcal{S}\subset \mathbb{C}_-$ (recall \eqref{defineCnegetive} for the definition of $\mathbb{C}_-$), having at most four elements, such that the spectrum of  $A_F$ satisfies
	\begin{equation}\label{sperTA}
		\sigma\left(A_F\right)\subset\{0\}\cup \mathcal{S}\cup \left(-\infty, -\frac{1}{\mu}\right]\cup \left\{\lambda \in \mathbb{C}_-\quad \left\vert \quad  \left\vert \lambda +\frac{1}{\mu}\right \vert =\frac{1}{\mu}\right\}\right..
	\end{equation}
\end{proposition}

\begin{proof}
	The proof can be completed following line by line the Theorem \ref{theo 2}. The only difference is that we  need to replace the matrix $\mathcal{M}_\lambda$ defined in \eqref{DMLAMBDA}  by
	\begin{align}
		\tilde{\mathcal{M}}_\lambda= \begin{bmatrix}
			\lambda\left(1+\frac{8a^3}{3}\right)+2a\left(\mu+\frac{1}{2a} +\frac{1}{\lambda}\right)+\frac{4a^2\lambda}{\omega_\lambda}&\ & -\lambda \left(1-\frac{4a^3}{3}\right)-2a\left(\mu +\frac{1}{\lambda}\right)\\
			-\lambda \left(1-\frac{4a^3}{3}\right)-2a\left(\mu +\frac{1}{\lambda}\right) &\ & \lambda\left(1+\frac{8a^3}{3}\right)+2a\left(\mu+\frac{1}{2a} +\frac{1}{\lambda}\right)+\frac{4a^2\lambda}{\omega_\lambda}
		\end{bmatrix},		\notag
	\end{align}
where $\omega_\lambda$ is defined in \eqref{DEOMGA}.
\end{proof}

We also need the result below, which is easy to check, so we skip the proof.

 \begin{proposition}\label{the_point_spectrum_AF}
With the  notation in Proposition \ref{anal_perturb}, we have that $0$ is not an eigenvalue of $A_F$. Moreover, for every $G\in \mathcal{L}(X, \mathbb{C})$, we have that $0\in \sigma(A+BG)$, which implies that system $\Sigma(A,B,C)$ is not exponentially stabilizable.
 \end{proposition}

%
%

The result below show that the system $\Sigma\left(A+BF,\,0,\, \begin{bmatrix} C\\F\end{bmatrix}\right)$ satisfies one of the main assumptions in Proposition \ref{mainusedtheorem}.

\begin{proposition}\label{strongly_stable}
	With the notation in Proposition \ref{anal_perturb}, the semigroup $\mathbb{T}^F$ generated by $A_F$ is strongly stable in $X$.
\end{proposition}

\begin{proof}
	We first remark that, since $\mathbb{T}^F$ is analytic, its growth bound $\omega_0\left(\mathbb{T}^F\right)$ is given (\cite[Proposition 2.9,\;p.120-121]{p2}) by
	\begin{equation*}
		\omega_0\left(\mathbb{T}^F\right)=\sup_{ \lambda \in \sigma(A_F)} {\rm Re}\,\lambda.
	\end{equation*}
	This fact, combined with Proposition \ref{thD1} and \ref{the_point_spectrum_AF}, implies that $\mathbb{T}^F$ is a bounded semigroup.
	
   On the other hand, from Proposition \ref{thD1} and \ref{the_point_spectrum_AF}, it follows that $A_F$ has no eigenvalues on the imaginary axis and that $\sigma\left(A_F\right)\cap i\mathbb{R}=\{0\}$.
Using the above properties and applying the Arendt-Batty theorem (see Arendt and Batty \cite{p22}), we obtain the announced conclusion.
\end{proof}

We are now in a position to prove the main result of this section, which can be seen as a reinforced version of Theorem \ref{thA1}, already
announced in Section \ref{sec_very_first}. Recall that the operators $A,\ B$ and $C$ appearing in the theorem below are those introduced of the beginning of this subsection.

\begin{theorem}\label{thA3}
	For every initial data satisfying the assumptions in Theorem \ref{th_exist_1}, there exists a unique optimal control $\bar{u}\in L^2(0,\infty)$ minimizing the cost functional $J(u)$ defined in \eqref{defineJu}. Moreover,  $\bar{u}$ can be expressed in the feedback form $\bar{u}(t)=-B^* P_0 z(t)$,
	where $P_0\in\mathcal{L}(X)$ is the minimal self-adjoint nonnegative solution to the algebraic Riccati equation
	$$A^* Pz+PAz-P B B^*Pz+C^* C z = 0 \qquad \quad (z\in \mathcal{D}(A)).$$
Finally, the corresponding minimal cost functional \eqref{defineJu} is given by
	$$J(\bar{u},\; z_0)=\left<P_0 z_0,\; z_0\right>_X.$$
\end{theorem}

\begin{proof}
As stated in Proposition \ref{echivalenta_mare}, system \eqref{A1}-\eqref{A_Last} is equivalent to system \eqref{A3}-\eqref{A4}. Hence, finding an optimal control for the latter system implies that the optimal control is also suitable for the original system.

To apply Proposition \ref{mainusedtheorem} we first remark that, according to Proposition \ref{strongly_stable}, the semigroup generated by $A_F=A+BF$ is strongly stable. Moreover, from Proposition \ref{leA1} it follows that
\begin{equation*}
	\int_{0}^{\infty} \left\vert \dot{H}(t) \right\vert^2 {\rm d}t \leq 2\Vert z_0 \Vert_{X}^2.
\end{equation*}
Combining the last estimate with \eqref{A3} and with the definitions \eqref{def_C_bun}, \eqref{DEF} of the operators $C$ and $F$,
it follows that the system $\Sigma\left(A+BF,\,0,\, \begin{bmatrix} C\\F\end{bmatrix}\right)$ is output stable. We can thus apply Proposition \ref{mainusedtheorem} to obtain the announced conclusions.
\end{proof}

\section{Appendix}\label{sec_app}
\subsection{Appendix A}\label{AppendixA}

This paragraph is devoted to the proof of Lemmas \ref{le1}-\ref{le4}.

\begin{proof}[Proof of Lemma \ref{le1}]
	
		Since
	\begin{equation}
		\frac{\lambda^2}{1+\mu\lambda}=
		\frac{1}{\mu^2}\frac{(\mu\lambda)^2}{\mu\lambda+1},\notag
	\end{equation}
	it clearly suffices  to prove \eqref{F1} and \eqref{F2} for $\mu=1$.
	
	We take $\lambda\neq -1$ and write $\lambda=\rho \exp({i\phi})$ with $\rho >0$. Since
	\begin{align}
		\frac{\lambda^2}{1+\lambda}=\frac{\rho^2\left(\cos(2\phi)+\rho \cos(\phi)+i(\sin(2\phi)+\rho \sin(\phi)\right)}{|1+\lambda|^2},\notag
	\end{align}
	it follows that $\frac{\lambda^2}{1+\lambda}\in\mathbb{R}_-$ if and only if
	\begin{equation}
		\cos(2\phi)+\rho \cos(\phi)<0 \quad {\rm and} \quad \sin(2\phi)+\rho \sin(\phi)=0.\notag
	\end{equation}
	The above condition implies that
	\begin{equation}\label{notag}
		\phi=\pi \quad {\rm and} \quad \rho >1 \quad {\rm or} \quad 2\cos(\phi)+\rho=0.
	\end{equation}
	The first condition in \eqref{notag} writes $\lambda\in (-\infty, -1)$. If the second condition in \eqref{notag} holds then
	\begin{align}
		&\vert 1+\lambda\vert^2\notag\\
		=&\left(1+\rho\cos(\phi)+i\rho\sin(\phi)\right)\left(1+\rho\cos(\phi)-i\rho\sin(\phi)\right)\notag\\
		=&\rho^2+2\rho\cos(\phi)+1\notag\\
		=&\rho(2\cos(\phi)+\rho)+1\notag\\
		=&1.\notag
	\end{align}
	So that the second condition is equivalent to $\vert \lambda +1 \vert=1$.
	
	\noindent We have thus proved the conclusion for $\mu=1$ and, consequently, for any positive $\mu$.
\end{proof}

\begin{proof}[Proof of Lemma \ref{le2}]
	 We first note that
	\begin{equation}
		\det \mathcal{M}_\lambda=4a^2\lambda\left(a+\frac{1}{\omega_\lambda}\right)\times \left[\lambda\left(2+\frac{4a^3}{3}\right)+4a\left(\mu +\frac{1}{\lambda }\right)+\frac{4a^2\lambda}{\omega_\lambda}\right]. \notag
	\end{equation}
	In order to obtain the conclusion of Lemma \ref{le2}, it suffices to check that the equation
	\begin{equation}
		4a^2\lambda\left[a+\frac{1}{\omega_\lambda}\right]\times \left[\lambda\left(2+\frac{4a^3}{3}\right)+4a\left(\mu +\frac{1}{\lambda }\right)+\frac{4a^2\lambda}{\omega_\lambda}\right]=0, \label{eqMM}
	\end{equation}
	admits at most four solutions satisfying  \eqref{C1} and lying in $\mathbb{C}_-$.
	To this aim, we begin by remarking that for any $\lambda$ satisfying \eqref{C1} we have
	\begin{equation}
		{\rm Re}\, \left(a+\frac{1}{\omega_\lambda}\right)\geq a \notag\,.
	\end{equation}
	Consequently, the complex number $\lambda$ satisfying \eqref{C1} is a solution of $\det\mathcal{M}_\lambda=0$, if and only if
	\begin{equation}
		\lambda\left(2+\frac{4a^3}{3}\right)+4a\left(\mu +\frac{1}{\lambda }\right)+\frac{4a^2\lambda}{\omega_\lambda}=0.\label{C3}
	\end{equation}

	Let now $\lambda \in \mathbb{C}^*$ with ${\rm Re}\,\lambda\geqslant 0$. Then $\lambda=\rho\exp({i\phi})$  with $\vert \phi\vert \leq \frac{\pi}{2}$ and $\rho > 0$. Then there exists an angle $\psi$ with $\vert \psi\vert\leq \frac{\pi}{2}$ such that
	\begin{equation}
		1+\mu\lambda=1+\mu\rho \cos(\phi)+ i\mu\rho \sin(\phi)=\vert 1+\mu\lambda\vert \exp({i\psi}). \notag
	\end{equation}
	Moreover, it is easy to see that $\psi\geq 0$ if and only if $\phi \geq 0$, thus $
		\left\vert\phi-\frac{\psi}{2}\right\vert\leq \frac{\pi}{2}$.
	Combining this estimate with the fact, following from the definition \eqref{DEOMGA} of $\omega_\lambda$,
that
$$
\omega_\lambda=\frac{\rho}{\vert 1+\mu\lambda\vert^{\frac{1}{2}} }\exp\left({i\left(\phi-\frac{\psi}{2}\right)}\right),
$$
it follows that
	\begin{equation}
		 {\rm Re}\,\left(\frac{\omega_\lambda}{\lambda}\right)\geq 0 \qquad\qquad(\lambda \in \mathbb{C}^*,\ \ {\rm Re}\,\lambda\geqslant 0).\notag
	\end{equation}
	The last inequality implies  that for every $\lambda\in \mathbb{C}^*$ with ${\rm Re}\, \lambda \geqslant 0$ we have
	\begin{equation}
		{\rm Re}\left(\lambda\left(2+\frac{4a^3}{3}\right)+4a\left(\mu +\frac{1}{\lambda }\right)+\frac{4a^2\lambda}{\omega_\lambda}\right)\geq 4a\mu >0.\notag
	\end{equation}
	We have thus shown that any the solution of \eqref{C3} either vanishes or it lies in the left half plane.\\
	Moreover, if $\lambda$ satisfying \eqref{C1} is a solution of $\det\mathcal{M}_\lambda=0$, then we have
	\begin{equation}
		\left[\lambda^2\left(2+\frac{4a^3}{3}\right)+4a\left(\mu\lambda +1\right)\right]^2=16a^4\lambda^2 (1+\mu\lambda),\notag
	\end{equation}
	so that $\det\mathcal{M}_\lambda$ can vanish for at most four values of $\lambda$ satisfying \eqref{C1}. 	
\end{proof}

\begin{proof}[Proof of Lemma \ref{le3}]
	Firstly, by the definition of $D_+(\omega)$ \eqref{defintion_of_operator_D_plus}, for every $\gamma_+\in \mathbb{C}$, we have
	\begin{equation}
		\int_{a}^{\infty}\vert (D_+(\omega)\gamma_+)(x)\vert^2{\rm d}x\leq \exp(2({\rm Re}\,\omega)a)\vert\gamma_+\vert^2 \int_{a}^{\infty} \exp\left(-2({\rm Re}\, \omega)x\right){\rm d}x= \frac{\vert\gamma_+\vert^2}{2{\rm Re}\,\, \omega},\notag
	\end{equation}
	which implies the second equality in \eqref{NBD}.
	
	On the other hand, to estimate the norm of operator $R_+(\omega)$, defined in \eqref{defintion_of_operator_R_plus}, we remark that
	\begin{align}\label{estimate_R_plus_L1norm1}
		&\left\vert
		\frac{1}{2\omega}\left(\int_{a}^{x}\exp(\omega\xi)\varphi(\xi){\rm d}\xi\right)\exp(-\omega x)
		\right \vert\notag \\
		\leq& \frac{1}{2\vert\omega\vert}\left(\int_{a}^{x}\exp\left(\left({\rm Re}\,\omega\right)\xi\right)\vert \varphi(\xi)\vert {\rm d}\xi\right)\exp\left(-({\rm Re}\,\omega)x\right)\notag\\
		\leq &
		\frac{\Vert \varphi \Vert_{L^\infty (a,\infty)}}{2\vert\omega\vert}\left(\int_{a}^{x}\exp(({\rm Re}\,\omega)\xi){\rm d}\xi\right)\exp(-({\rm Re}\,\omega)x)\notag\\
		=&
		\frac{\Vert \varphi \Vert_{L^\infty (a,\infty)}}{2\vert\omega\vert}\frac{\exp(({\rm Re}\,\omega)x)-\exp(({\rm Re}\,\omega)a)}{{\rm Re}\,\omega}\exp(-({\rm Re}\, \omega)x)\notag\\
		\leq&
		\frac{\Vert \varphi \Vert_{L^\infty (a,\infty)}}{2\vert\omega\vert {\rm Re}\, \omega} \qquad\qquad\quad (\varphi\in L^\infty(a,\infty),\; x\in (a,\infty)).
	\end{align}
	With a similar calculation for the other two terms of $R_+(\omega)$, we obtain
	\begin{equation}\label{estimate_R_plus_L1norm2}
		\left \vert \frac{1}{2\omega}\left(\int_{x}^{\infty}\exp(-\omega \xi)\varphi(\xi){\rm d}\xi\right)\exp(\omega x)\right \vert
		\leq
		\frac{\Vert \varphi \Vert_{L^\infty (a,\infty)}}{2\vert\omega\vert {\rm Re}\, \omega},
	\end{equation}
	and
	\begin{equation}\label{estimate_R_plus_L1norm3}
		\left\vert\frac{1}{2\omega}\left(\int_{a}^{\infty}\exp(-\omega \xi)\varphi(\xi){\rm d}\xi\right)\exp(2\omega a)\exp(-\omega x)\right\vert
		\leq
		\frac{\Vert \varphi \Vert_{L^\infty (a,\infty)}}{2\vert\omega\vert {\rm Re}\, \omega},
	\end{equation}
	where
	\begin{equation*}
		\varphi\in L^\infty(a,\infty) \qquad {\rm and}\qquad  x\in (a,\infty).
	\end{equation*}
	Combining the defintion \eqref{defintion_of_operator_R_plus} of $R_+(\omega)$ with the \eqref{estimate_R_plus_L1norm1}-\eqref{estimate_R_plus_L1norm3}, it follows that
	\begin{equation}
		\left\Vert R_+(\omega)\right\Vert_{\mathcal{L}(L^\infty (a,\infty))}\leq\frac{3}{2\vert\omega\vert {\rm Re}\, \omega} \qquad\qquad ({\rm Re}\, \omega>0,\;\varphi\in L^\infty(a,\infty)).\label{ERB1}
	\end{equation}
	We next estimate the norm in $\mathcal{L}(L^1(a,\infty))$ of the operator $R_+(\omega)$ defined in \eqref{defintion_of_operator_R_plus}. For the first term in the right hand side of \eqref{defintion_of_operator_R_plus} we can use Tonelli's theorem to get
	\begin{align}\label{estimate_of_R_plus1}
		&\int_{a}^{\infty}\left\vert\frac{1}{2 \omega}\left(\int_{a}^{x}\exp(\omega \xi) \varphi(\xi){\rm d}\xi\right)\exp(-\omega x)\right\vert {\rm d}x\notag\\
		\leq&
		\frac{1}{2\vert \omega\vert}\int_{a}^{\infty}\left(\int_{a}^{x}\exp(({\rm Re}\,\omega)\xi)\vert\varphi(\xi)\vert {\rm d}\xi\right)\exp(-({\rm Re}\,\omega)x){\rm d}x\notag\\
		=&\frac{1}{2\vert \omega\vert }\int_{a}^{\infty}\left(\int_{\xi}^{\infty}\exp(-({\rm Re}\,\omega)x){\rm d}x\right) \exp(({\rm Re}\,\omega)\xi)\vert\varphi(\xi)\vert {\rm d}\xi\notag\\
		\leq&
		\frac{\Vert \varphi \Vert_{L^1 (a,\infty)}}{2\vert\omega\vert {\rm Re}\, \omega}\qquad\qquad\qquad (\varphi\in L^1(a,\infty)).
	\end{align}
	In a fully similar manner it can be checked that, for every $\varphi\in L^1(a,\infty)$, the other terms in the right hand side of \eqref{defintion_of_operator_R_plus} satisfy the estimates,
	\begin{equation}\label{estimate_of_R_plus2}
		\int_{a}^{\infty}\left\vert	\frac{1}{2 \omega} \left(\int_{x}^{\infty}\exp(-\omega \xi) \varphi(\xi){\rm d}\xi\right)\exp(\omega x)\right\vert {\rm d}x
		\leq
		\frac{\Vert \varphi \Vert_{L^1 (a,\infty)}}{2\vert\omega\vert {\rm Re}\, \omega},
	\end{equation}
	and
	\begin{equation}\label{estimate_of_R_plus3}
		\int_{a}^{\infty}\left\vert\frac{1}{2 \omega}\left( \int_{a}^{\infty}\exp(-\omega \xi) \varphi(\xi){\rm d}\xi\right)\exp(2\omega a)\exp(-\omega x)\right\vert {\rm d}x\leq \frac{\Vert \varphi \Vert_{L^1 (a,\infty)}}{2\vert\omega\vert {\rm Re}\, \omega}.
	\end{equation}
	Combining the definition \eqref{defintion_of_operator_R_plus} of $R_+(\omega)$ with formulas \eqref{estimate_of_R_plus1}-\eqref{estimate_of_R_plus3}, we obtain that
	\begin{equation}
		\Vert R_+(\omega)\Vert_{\mathcal{L}(L^1(a,\infty))}\leq \frac{3}{2\vert\omega \vert {\rm Re}\,\omega} \qquad\qquad\qquad ({\rm Re}\,\omega>0).\label{ERB2}
	\end{equation}
	Using the \eqref{ERB1} and \eqref{ERB2}, combined with the Riesz-Thorin theorem (see, for instance, \cite[p.318]{p24}), we find the second inequality in \eqref{NRB} also holds.
	
	The  first inequalities in \eqref{NBD} and \eqref{NRB} can be proved in a fully similar manner, so we skip the corresponding details.
\end{proof}

\begin{proof}[Proof of Lemma \ref{le4}]
	
	After a standard caculation, we find that the general solution of equation \eqref{Eqqlad} is
	\begin{align}
		q_\lambda (x)=&\frac{1}{2\omega_\lambda}\left(\int_{a}^{x}\exp(\omega_\lambda\xi)\varphi(\xi){\rm d}\xi\right)\exp(-\omega_\lambda x)-\frac{1}{2\omega_\lambda}\left(\int_{a}^{x}\exp(-\omega_\lambda\xi)\varphi(\xi){\rm d}\xi\right)\exp(\omega_\lambda x)\notag\\
		&+K_1\exp(-\omega_\lambda x)+K_2\exp(\omega_\lambda x)\qquad\qquad\qquad\qquad (x\geq a),\notag
	\end{align}
	where $K_1$ and $K_2$ are complex constants. \\
	\indent Combined Lemma \ref{le1} and \ref{le2}, we get ${\rm Re}\, \omega_\lambda>0$, where $\omega_\lambda$ is defined by \eqref{DEOMGA}. In order to have $q_\lambda \in L^2(\mathcal{E})$  we necessarily choose
	\begin{equation*}
		K_2=\frac{1}{2\omega_\lambda}\left(\int_{a}^{\infty}\exp(-\omega_\lambda\xi)\varphi(\xi){\rm d}\xi\right),
	\end{equation*}
	so that
	\begin{align*}
		q_\lambda (x)=&\frac{1}{2\omega_\lambda}\left(\int_{a}^{x}\exp(\omega_\lambda\xi)\varphi(\xi){\rm d}\xi\right)\exp(-\omega_\lambda x)
		+\frac{1}{2\omega_\lambda}\left(\int_{x}^{\infty}\exp(-\omega_\lambda\xi)\varphi(\xi){\rm d}\xi\right)\exp(\omega_\lambda x)\notag\\
		&+K_1\exp(-\omega_\lambda x) \qquad\qquad\qquad\qquad (x\geq a).
	\end{align*}
	Using next the boundary condition in \eqref{Eqqlad} we obtain that
	\begin{equation}
		\gamma_+=\frac{1}{2\omega_\lambda}\left(\int_{a}^{\infty}\exp(-\omega_\lambda\xi)\varphi(\xi){\rm d}\xi\right)\exp (\omega_\lambda a)+K_1\exp(-\omega_\lambda a), \notag
	\end{equation}
	which yields
	\begin{equation}
		K_1=\gamma_+ \exp(\omega_\lambda a)-\frac{1}{2\omega_\lambda}
		\left(\int_{a}^{\infty}\exp(-\omega_\lambda\xi)\varphi(\xi){\rm d}\xi\right)\exp(2\omega_\lambda a).
	\end{equation}
	With the above obtained values of $K_1$ and $K_2$ we see that if $q_\lambda \in H^2(\mathcal{E})$  satisfies \eqref{Eqqlad}
	then $q_\lambda$ satisfies the second equation in \eqref{whole_whale}. By a fully similar calculation we can check that
	if $q_\lambda \in H^2(\mathcal{E})$  satisfies \eqref{Eqqlad}
	then $q_\lambda$ satisfies the first equation in \eqref{whole_whale}. Thus the unique possible solution $q_\lambda \in H^2(\mathcal{E})$  of
	\eqref{Eqqlad} is given by \eqref{whole_whale}.
	
	On the other hand, we know from  Lemma \ref{le3} (respectively from the above calculations) that $q_\lambda$ defined by \eqref{whole_whale} lies in $L^2(\mathcal{E})$ (respectively that it satisfies \eqref{Eqqlad}). This implies that indeed  $q_\lambda$ defined by \eqref{whole_whale} lies in $H^2(\mathcal{E})$ and it is the unique solution in this space of \eqref{Eqqlad}.
\end{proof}

\subsection{Appendix B}\label{Appendix B}

This subsection is devoted to the proofs of the lemmas in Section \ref{sec_well_new}.

\begin{proof}[Proof of Lemma \ref{le5}]
	By obvious homogeneity reasons (see the proof of Lemma \ref{le1}),
	it suffices to prove the conclusion of the present lemma for $\mu=1$. In this case we have
	\begin{equation*}
		\vert 1+\lambda \vert^2=1+\rho^2+2\rho\cos(\phi),
	\end{equation*}
	so that
	\begin{align}
		\frac{\lambda^2}{1+\lambda}&=\frac{\left(\rho^2\cos(2\phi)+i\rho\sin(2\phi)\right)\left(1+\rho\cos(\phi)-i\rho\sin(\phi)\right)}{1+\rho^2+2\rho\cos(\phi}\notag\\
		&=\frac{\rho^2\cos(2\phi)+\rho^3\cos(\phi)+i(\rho^2\sin(2\phi)+\rho^3\sin(\phi))}{1+\rho^2+2\rho\cos(\phi)} \qquad\qquad\quad (\lambda\in\Sigma_\theta). \label{F3}
	\end{align}
	Moreover, we have that
	\begin{equation*}
		\rho^2\sin(2\phi)+\rho^3\sin(\phi)=\rho^2\sin(\phi)(2\cos(\phi)+\rho)\neq 0,
	\end{equation*}
	where
	\begin{equation*}
		\phi\in
		\left(
		-\frac{\pi}{2}-\theta,\; \frac{\pi}{2}+\theta\right)\verb|\|\{0\} \qquad  {\rm and} \qquad \rho\geq\frac{4}{1-\sin(\theta)}
		.
	\end{equation*}
	The above three formulas imply that the imaginary part of $\frac{\lambda^2}{1+\lambda}$, with $\lambda=\rho \exp(i\phi)$ and $\rho\geq \frac{4}{1-\sin (\theta)}$ vanishes iff $\phi=0$. On the other hand, formula \eqref{F3} clearly yields that ${\rm Re}\,\frac{\lambda^2}{1+\lambda}>0$. We have thus proved that \eqref{F1} holds.

	We still prove that formula \eqref{F2} holds for $\mu=1$. To  this aim, we first recall the classical formula
	\begin{equation}\label{estimateofReomega}
		{\rm Re}\,\sqrt{z}=\frac{1}{\sqrt{2}}\sqrt{\vert z\vert + {\rm Re}\,z} \qquad\qquad\quad (z\in\mathbb{C}\verb|\|\mathbb{R}_-).
	\end{equation}
	We next apply the above formula for $z=\frac{\lambda^2}{1+\lambda}$ and $\lambda=\rho \exp(i\phi)$, with
	\begin{equation}
		\rho\geq \frac{4}{1-\sin (\theta)} \qquad	{\rm and}\qquad \phi\in
		\left(
		-\frac{\pi}{2}-\theta,\, \frac{\pi}{2}+\theta\right).\label{the_condition_of_lambda}
	\end{equation}
	Combining \eqref{estimateofReomega}, \eqref{the_condition_of_lambda} and \eqref{F3}, we obtain that
	\begin{align}
		\label{F5} 2\left[{\rm Re}\,\left(\sqrt{\frac{\lambda^2}{1+\lambda}}\right)\right]^2
		&=\frac{\rho^2\sqrt{1+\rho^2+2\rho\cos(\phi)}}{1+\rho^2+2\rho\cos(\phi)}+\frac{\rho^2\cos(2\phi)+\rho^3\cos(\phi)}{1+\rho^2+2\rho\cos(\phi)}\notag\\
		&=\frac{\rho^2}{1+\rho^2+2\rho\cos(\phi)}\left(\sqrt{1+\rho^2+2\rho\cos(\phi)}+\cos(2\phi)+\rho\cos(\phi)\right),
	\end{align}
	for every $\lambda=\rho \exp(i\phi)$ with $\rho$ and $\phi$ satisfying \eqref{the_condition_of_lambda}. Moreover, since $\rho>1$, we have
	\begin{align*}
		&\sqrt{1+\rho^2+2\rho\cos(\phi)}+\cos(2\phi)+\rho\cos(\phi)\notag\\
		\geq& \rho-1+\cos(2\phi)+\rho\cos(\phi)\notag\\
		\geq 	&\rho(1-\sin(\theta))-2\notag\\
		\geq	&\frac{1}{2}\rho(1-\sin(\theta)).
	\end{align*}
	Using the above equation and \eqref{F5}, for every $\lambda=\rho \exp(i\phi)$ satisfying \eqref{the_condition_of_lambda}, it follows that
	\begin{equation*}
		2\left[{\rm Re}\,\left(\sqrt{\frac{\lambda^2}{1+\lambda}}\right)\right]^2\geqslant
		\frac{\rho^3(1-\sin(\theta))}{2(1+\rho^2+2\rho\cos(\phi))}\geq\frac{\rho^3(1-\sin(\theta))}{2(1+\rho)^2}\geq\frac{\rho(1-\sin(\theta))}{8},
	\end{equation*}
	which ends the proof of Lemma \ref{le5}.
\end{proof}
\begin{proof}[Proof of Lemma \ref{le6}]
	From Lemma \ref{le2}, we find for every $\lambda$ satisfying \eqref{F6}, the matrix $\mathcal{M}_\lambda$ defined in \eqref{DMLAMBDA} is invertible. Moreover, from the definition \eqref{DMLAMBDA} of $\mathcal{M}_\lambda$ , we have
	\begin{equation*}
		\lambda^{-1}\mathcal{M}_\lambda=\left[\begin{gathered}
			1+\frac{8a^3}{3}\qquad\quad -1+\frac{4a^3}{3}\\
			-1+\frac{4a^3}{3}\qquad\quad 1+\frac{8a^3}{3}
		\end{gathered}\right]+2a\left(\frac{\mu}{\lambda}+\frac{1}{\lambda^2}\right)\left[\begin{gathered}
			1\quad -1\\
			-1\quad 1
		\end{gathered}\right]+\frac{4a^2}{\omega_\lambda}\left[\begin{gathered}
			1\quad 0\\
			0\quad 1
		\end{gathered}\right].
	\end{equation*}
	The above formula can be rewritte,
	\begin{equation*}
		\lambda^{-1}\mathcal{M}_\lambda=\left[\begin{gathered}
			1+\frac{8a^3}{3}\qquad -1+\frac{4a^3}{3}\\
			-1+\frac{4a^3}{3}\qquad 1+\frac{8a^3}{3}
		\end{gathered}\right]+\mathcal{P}_\lambda,
	\end{equation*}
where
$$
\mathcal{P}_\lambda=2a\left(\frac{\mu}{\lambda}+\frac{1}{\lambda^2}\right)\left[\begin{gathered}
			1\quad -1\\
			-1\quad 1
		\end{gathered}\right]+\frac{4a^2}{\omega_\lambda}\left[\begin{gathered}
			1\quad 0\\
			0\quad 1
		\end{gathered}\right].
$$
Using next Lemma \ref{le5} we see that there exists a constant $\tilde{K}(\theta,\,\mu)>0$ such that
	\begin{equation*}
		  \Vert \mathcal{P}_\lambda \Vert \leq \frac{\tilde{K}(\theta,\,\mu)}{\sqrt{\vert \lambda \vert}}\qquad\qquad\qquad
			\left(\lambda \in \Sigma_\theta,\; \vert\lambda\vert\geq R(a,\,\mu,\,\theta)\right).
	\end{equation*}
	Moreover, after a simple calculation, we see that the inverse of the matrix $M$ defined in \eqref{AM} is
	\begin{align*}
		M^{-1}=\left[\begin{gathered}
			1+\frac{8a^3}{3} \qquad\quad -\left(1-\frac{4a^3}{3}\right)\\
			-\left(1-\frac{4a^3}{3}\right)  \qquad\quad 1+\frac{8a^3}{3}
		\end{gathered}
		\right],
	\end{align*}
	Consequently,
	\begin{align}
		&\Vert \lambda\mathcal{M}^{-1}_\lambda\Vert \notag\\
		=&\left\Vert (M^{-1}+\mathcal{P}_\lambda)^{-1}\right\Vert
		\notag\\
		\leq& \left\Vert M^{-1}\right\Vert\left\Vert \left(\mathbb{I}+\frac{\mathcal{P}_\lambda}{M^{-1}}\right)\right\Vert\notag\\
		\leq& \left\Vert M^{-1}\right\Vert \left\Vert\frac{1}{1-\left\Vert \frac{\mathcal{P}_\lambda}{\left\Vert M^{-1}\right\Vert}\right\Vert}\right\Vert \qquad\qquad \left(	\lambda \in \Sigma_\theta,\; \vert\lambda\vert\geq R(a,\,\mu,\,\theta)\right).
	\end{align}
	The above formula implies the inequality \eqref{F7}.
\end{proof}
{\bf Acknowledgments.}  
Funded by the European Union (Horizon Europe MSCA project Modconflex, grant number 101073558).

The authors are grateful to Debayan Maity, Jorge San Mart\'in and Tak\'eo Takahashi for numerous suggestions, in particular those involving the detailed estimates from Section \ref{sec_app}.

\bibliographystyle{siam}
\bibliography{references}

\end{document}